\documentclass[10pt]{amsart}
\usepackage{amsmath,amssymb,latexsym,esint,cite,mathrsfs}
\usepackage{verbatim,wasysym}
\usepackage[left=2.6cm,right=2.7cm,top=2.9cm,bottom=2.8cm]{geometry}
\usepackage{tikz,enumitem,graphicx, subfig, microtype, color}
\usepackage{epic,eepic}

\usepackage[colorlinks=true,urlcolor=blue,citecolor=red,linkcolor=blue,
linktocpage,pdfpagelabels, bookmarksnumbered,bookmarksopen]{hyperref}
\usepackage[hyperpageref]{backref}
\usepackage[english]{babel}

\numberwithin{equation}{section}

\newtheorem{thm}{Theorem}[section]
\newtheorem{lem}[thm]{Lemma}

\newtheorem{Prop}[thm]{Proposition}

\newcommand{\R}{\mathbb{R}}

\def\al {\alpha}

\begin{document}
\title[Nonlocal equations]{On elliptic equations with Stein-Weiss type convolution parts}

\author[L.\ Du]{Lele Du}
\author[F.\ Gao]{Fashun Gao}
\author[M.\ Yang]{Minbo Yang$^*$}

\address{Lele Du,\ \  Minbo Yang,\newline\indent
Department of Mathematics, Zhejiang Normal University, \newline\indent
	321004, Jinhua, Zhejiang, P. R. China}
\email{L. Du:1182714397@qq.com;\ \  M. Yang:mbyang@zjnu.edu.cn}

\address{Fashun Gao, \newline\indent Department of Mathematics and Physics, Henan University of Urban Construction, \newline\indent
 467044, Pingdingshan, Henan,  P. R. China}
\email{fsgao@zjnu.edu.cn}


\thanks{$^*$Minbo Yang is the corresponding author who was partially supported by NSFC(11971436, 12011530199) and ZJNSF(LD19A010001) and Fashun Gao was partially supported by NSFC(11901155).}
\begin{abstract}
The aim of this paper is to study the critical elliptic equations with Stein-Weiss type convolution parts
$$
\displaystyle-\Delta u
=\frac{1}{|x|^{\alpha}}\left(\int_{\mathbb{R}^{N}}\frac{|u(y)|^{2_{\alpha, \mu}^{\ast}}}{|x-y|^{\mu}|y|^{\alpha}}dy\right)
|u|^{2_{\alpha, \mu}^{\ast}-2}u,~~~x\in\mathbb{R}^{N},
$$
where the critical
exponent is due to the weighted Hardy-Littlewood-Sobolev inequality and Sobolev embedding. We develop a nonlocal version of concentration-compactness principle to investigate the existence of solutions and study the regularity, symmetry of positive solutions by moving plane arguments. In the second part, the subcritical case is also considered, the existence, symmetry, regularity of the positive solutions are obtained.

\end{abstract}

\maketitle



%
%
%


\section{Introduction and main results}

The existence of sharp constant for the Stein-Weiss inequality, also called the weighted Hardy-Littlewood-Sobolev inequality, has attracted a lot of interests.
\begin{Prop}\label{prop1}\cite{SW}
 Let $1<r,s<\infty$, $0<\mu<N$, $\alpha+\beta\geq0$ and $\alpha+\beta+\mu\leq N$, $f\in L^{r}(\mathbb{R}^N)$ and $h\in L^{s}(\mathbb{R}^N)$. There exists a constant $C(\alpha,\beta,\mu,N,s,r)$, independent of $f,h$, such that
\begin{equation}\label{WHLS}
\int_{\mathbb{R}^{N}}\int_{\mathbb{R}^{N}}\frac{f(x)h(y)}{|x|^{\alpha}|x-y|^{\mu}|y|^{\beta}}dxdy\leq C(\alpha,\beta,\mu,N,s,r) |f|_{r}|h|_{s}.
\end{equation}
where
$$
1-\frac{1}{r}-\frac{\mu}{N}<\frac{\alpha}{N}<1-\frac{1}{r} \ \ \mbox{and} \ \ \frac{1}{r}+\frac{1}{s}+\frac{\alpha+\beta+\mu}{N}=2.
$$
\end{Prop}
\noindent
 In 1983, Lieb \cite{LE1} proved the existence of sharp constant for the case that one of $r$ and $s$ equals to $2$ or $r=s$. For $1<r,s<\infty$ with $\frac1r+\frac1s=1$, the sharp constant is given by Beckner in \cite{B1, B2}. Dou and Zhu \cite{DZ} established the reversed HLS inequality and classified the extremal functions for the best constant.

Concerning the weighted Hardy-Littlewood-Sobolev inequality \eqref{WHLS}, if $\alpha=\beta$ and $s=r$, the integral
$$
\int_{\mathbb{R}^{N}}\int_{\mathbb{R}^{N}}\frac{|u(x)|^{q}|u(y)|^{q}}{|x|^{\alpha}|x-y|^{\mu}|y|^{\alpha}}dxdy
$$
is well defined if
$$
2-\frac{2\alpha+\mu}{N}\leq q\leq\frac{2N-2\alpha-\mu}{N-2}.
$$
In this sense, we call $2_{\ast \alpha, \mu}=2-\frac{2\alpha+\mu}{N}$ the lower critical exponent and $2_{\alpha, \mu}^{\ast}=\frac{2N-2\alpha-\mu}{N-2}$ the upper critical exponent in the sense of the weighted Hardy-Littlewood-Sobolev inequality. Generally, for $\alpha=\beta\geq0$ and $0<2\alpha+\mu\leq N$, the limit embedding for the upper critical exponent leads to the inequality
$$
\left(\int_{\mathbb{R}^{N}}\int_{\mathbb{R}^{N}}\frac{|u(x)|^{2_{\alpha, \mu}^{\ast}}|u(y)|^{2_{\alpha, \mu}^{\ast}}}{|x|^{\alpha}|x-y|^{\mu}|y|^{\alpha}}dxdy\right)^{\frac{1}{2_{\alpha, \mu}^{\ast}}}\leq C\int_{\mathbb{R}^N}|\nabla u|^{2}dx.
$$
We define
$$
\|u\|_{\alpha, \mu}^{2\cdot2_{\alpha, \mu}^{\ast}}:=\int_{\mathbb{R}^N}\int_{\mathbb{R}^N}
\frac{|u(x)|^{2_{\alpha, \mu}^{\ast}}|u(y)|^{2_{\alpha, \mu}^{\ast}}}
{|x|^{\alpha}|x-y|^{\mu}|y|^{\alpha}}dxdy
$$
and use $S_{\alpha, \mu}$ to denote the best constant
\begin{equation}\label{MiniPro}
S_{\alpha, \mu}:=\inf\limits_{u\in D^{1,2}(\mathbb{R}^N)\backslash\{{0}\},\|u\|_{\alpha}=1} \int_{\mathbb{R}^N}|\nabla u|^{2}dx.
\end{equation}
From the weighted Hardy-Littlewood-Sobolev inequality, for all $u\in D^{1,2}(\mathbb{R}^N)$ we know
$$
\|u\|_{\alpha,\mu}^{2}\leq C(N,\mu,\alpha)^{\frac{1}{2_{\alpha, \mu}^{\ast}}}|u|_{2^{\ast}}^{2}.
$$
Then
$$
S_{\alpha, \mu}\geq\frac{S}{C(N,\mu,\alpha)^{\frac{1}{2_{\alpha, \mu}^{\ast}}}}>0,
$$
where $S$ is the best Sobolev constant. Obviously, the study of the best constant $S_{\alpha, \mu}$ is related to the nonlocal Euler-Lagrange equation:
\begin{equation}\label{Intcase}
\displaystyle-\Delta u
=\frac{1}{|x|^{\alpha}}\left(\int_{\mathbb{R}^{N}}\frac{|u(y)|^{2_{\alpha, \mu}^{\ast}}}{|x-y|^{\mu}|y|^{\alpha}}dy\right)
|u|^{2_{\alpha, \mu}^{\ast}-2}u,~~~x\in\mathbb{R}^{N},
\end{equation}
where the convolutionary nonlinearity is called the Stein-Weiss type.

For the case $\alpha=\beta=0$ in Proposition \ref{prop1}, Lieb \cite{LE1} proved that the best constant for the classical Hardy-Littlewood-Sobolev inequality can be achieved by some extremals, that is
\begin{Prop}\label{prop1}
If $\alpha=\beta=0$ and $s=r=2N/(2N-\mu)$, then there is equality in \eqref{WHLS} if and only if $f\equiv(const.)h$ and
$$
h(x)=A(\gamma^{2}+|x-a|^{2})^{-(2N-\mu)/2}
$$
for some $A\in \mathbb{C}$, $0\neq\gamma\in\mathbb{R}$ and $a\in \mathbb{R}^{N}$.
\end{Prop}
\noindent
Then he posed the classification of the solutions of
\begin{equation}\label{Int}
u(x)=\int_{\mathbb{R}^{N}}\frac{u(y)^{\frac{N+\tau}{N-\tau}}}{|x-y|^{N-\tau}}dy,~~x\in\mathbb{R}^{N}
\end{equation}
as an open problem. In fact, equation \eqref{Int} arises as an Euler-Lagrange equation for a functional under a constraint in the
context of the Hardy-Littlewood-Sobolev inequality and is closely related to the well-known fractional equation
\begin{equation}\label{Frac}
(-\Delta)^{\frac{\tau}{2}}u=u^{\frac{N+\tau}{N-\tau}},~~x\in\mathbb{R}^{N}.
\end{equation}
When $N\geq3$, $\tau=2$, equation \eqref{Frac} goes back to
\begin{equation}\label{lcritical}
-\Delta u=u^{\frac{N+2}{N-2}},~~x\in\mathbb{R}^{N}.
\end{equation}
The classification of the solutions of equation \eqref{lcritical} and the related best Sobolev constant play an important role in the Yamabe problem, the prescribed scalar curvature problem on Riemannian manifolds and the priori estimates in nonlinear equations. Aubin \cite{A}, Talenti \cite{Ta} proved that the best Sobolev constant $S$ can be achieved by a two-parameter solution of the form
\begin{equation}\label{U0}
 U_0(x):=[N(N-2)]^{\frac{N-2}{4}}\Big(\frac{t}{t^2+|x-\xi|^{2}}\Big)^{\frac{N-2}{2}}.
 \end{equation}
 Equation \eqref{lcritical} is related to the Euler-Lagrange equation
of the extremal functions of the Sobolev inequality  and is a special case of the Lane-Emden equation
\begin{equation}\label{LE}
-\Delta u=u^{p},~~x\in\mathbb{R}^{N}.
\end{equation}
It is well known that, for $1\leq p<\frac{N+2}{N-2}$, Gidas and Spruck \cite{GS} proved that equation \eqref{LE} has no positive solutions. This
result is optimal in the sense that for any $p\geq\frac{N+2}{N-2}$, there are infinitely many positive solutions
to \eqref{LE}. Gidas, Ni and Nirenberg \cite{GNN}, Caffarelli, Gidas and Spruck \cite{CGS} proved the symmetry and uniqueness of the positive solutions respectively. Chen and Li \cite{CL}, Li \cite{LC} simplified the results above as an application of the moving plane method.  Wei and Xu \cite{WX} generalized the classification of the solutions of the more general equation \eqref{Frac} with $\tau$ being any even number between $0$ and $N$. Later, Chen, Li and Ou \cite{CLO1} developed the method of moving planes in integral forms to prove that any critical points of the functional was radially symmetric and assumed the unique form and gave a positive answer to Lieb's open problem involving \eqref{Int}. Li \cite{Li} also studied the regularity of the locally integrable solution for \eqref{Int} and used moving sphere method to prove the classification. In \cite{CLO3, Li}, the authors considered the nonnegative solutions of the integral equation
$$
u(x)=\int_{\mathbb{R}^{N}}\frac{|u(y)|^{\frac{2\tau}{N-\tau}}u(y)}{|x-y|^{N-\tau}}dy,
$$
and proved that $u\in  C^{\infty}(\mathbb{R}^{N})$.
As a generalization of equation \eqref{Int}, Lu and Zhu \cite{LZ} recently studied the symmetry and
regularity of extremals of the following weighted integral equation:
\begin{equation}\label{Int2}
u(x)=\int_{\mathbb{R}^{N}}\frac{u(y)^{\frac{(N+\tau-2s)}{N-\tau}}}{|y|^s|x-y|^{N-\tau}}dy,~~x\in\mathbb{R}^{N},
\end{equation}
which is related to the fractional singular case
\begin{equation}\label{Frac2}
(-\Delta)^{\frac{\tau}{2}}u=\frac{u^{\frac{N+\tau-2s}{N-\tau}}}{|x|^s},~~x\in\mathbb{R}^{N}.
\end{equation}
If $\tau=2$ and $0<s<2$, equation \eqref{Frac2} is closely related to the Euler-Lagrange equation
of the extremal functions of the Hardy-Sobolev inequality which says that there exists a constant $C>0$ such that
\begin{equation}\label{HS}
\left(\int_{\mathbb{R}^{N}}\frac{u^{\frac{2(N-s)}{N-\tau}}}{|x|^s}dx\right)^{\frac{N-\tau}{N-s}}\leq C \int_{\mathbb{R}^{N}}|\nabla u|^{2}dx.
\end{equation}
Lieb \cite{LE1} proved that the best constant in \eqref{HS} is achieved and the extremal
function is identified by
$$\frac{1}{(1+|x|^{2-t})^{\frac{N-2}{2-t}}}.$$

Now we are interested in the existence of extremal functions for the best constant $S_{\alpha, \mu}$ and to classify the positive solutions for equation \eqref{Intcase}. The best constant for the weighted Hardy-Littlewood-Sobolev inequality has already been proved by Lieb in \cite{LE1}. Chen and Li \cite{CLP} proved the uniqueness of the solutions for the following system, which is related to the Euler-Lagrange equation for the of inequality \eqref{WHLS},
\begin{equation}\label{PDE}
	\left\lbrace
	\begin{aligned}
		-\Delta(|x|^{\al}u(x))&=\frac{v^{q}(x)}{|x|^{\beta}},\\
		-\Delta(|x|^{\beta}v(x))&=\frac{u^{p}(x)}{|x|^{\al}},
	\end{aligned}
	\right.
\end{equation}
they also classified the solutions and obtained the best constant in the corresponding weighted HLS inequality when $\al=\beta$ and $p=q$. The existence of extremals for the best constant $S_{\alpha, \mu}$ is totally different and new. For the special case $\al=\beta=0$,
Liu\cite{Liu}, Lei \cite{YL2},  Du and Yang \cite{DY} studied equation \eqref{Intcase} and discussed the regularity, symmetry and classification of the positive solutions of equation \eqref{Intcase}. Furthermore, Du and Yang \cite{DY} obtained the nondegeneracy of the unique solutions for the equation when
$\mu$ close to $N$.

 As we all know, to study the best constant for critical minimizing problem, P.L. Lions established the well known Concentration-compactness principles  \cite{Ls1, Ls2, Ls3, Ls4}. Particularly, the second concentration-compactness principle \cite{Ls1} has also been developed to study the limit case involving the Hardy-Littlewood-Sobolev inequality. Inspired by \cite{Ls1}, we will prove a version of the second concentration-compactness principle for the weighted critical nonlocal problem and prove the existence of minimizers for problem \eqref{MiniPro}. The first result is the following theorem.
\begin{thm}\label{thm1.2}
Let $\{u_{n}\}\subset D^{1,2}(\mathbb{R}^N)$ be a minimizing sequence satisfying \eqref{MiniPro}. Then there exists a sequence $\{\tau_{n}\}\subset (0,+\infty)$ such that $\{u_{n}^{(\tau_n)}:=\tau^{\frac{N-2}{2}}u_{n}(\tau_n x)\}$ contains a convergent subsequence. In particular, there exists a minimizer for $S_{\alpha, \mu}$.
\end{thm}
After rescaling, we know there exists $u$ satisfies the nonlocal Euler-Lagrange equation \eqref{Intcase}.
Since we are working with equation \eqref{Intcase} in $D^{1,2}(\mathbb{R}^N)$, we know $u$ must be in $L^{2^*}(\mathbb{R}^{N}), 2^*=\frac{2N}{N-2}$. With this assumption, applying regularity lifting lemma by contracting operators, we can show that the positive solution $u$ possesses higher regularity if the parameters $\alpha, \mu$ are in suitable range.
\begin{thm}\label{QWEQ}
Assume that $N\geq3$, $\alpha\geq0$, $0<\mu<N$, $0<2\alpha+\mu\leq N$. Let $u\in D^{1,2}(\mathbb{R}^N)$ be a positive solution of equation \eqref{Intcase},

$(C1)$.~If~$N=3,4,5,6$ and $N-2\leq2\alpha+\mu\leq\min\{N,4\}$, then $u\in L^{p}(\mathbb{R}^{N})$ with
$$
p\in\left(\frac{N}{N-2},+\infty\right).
$$

$(C2)$.~If~$N=5,6$ and $4<2\alpha+\mu\leq N$ while $N\geq7$ and $N-2\leq2\alpha+\mu\leq N$, then $u\in L^{p}(\mathbb{R}^{N})$ with
$$
p\in\left(\frac{N}{N-2},\frac{2N}{2\alpha+\mu-4}\right).
$$

$(C3)$.~If~$N=3,4,5,6$ and $0<2\alpha+\mu<N-2$ while $N\geq7$ and $0\leq2\alpha+\mu\leq 4$ or $\frac{N+2}{2}\leq2\alpha+\mu<N-2$, then $u\in L^{p}(\mathbb{R}^{N})$ with
$$
p\in\left(\frac{2N}{N-2+2\alpha+\mu},\frac{2N}{N-2-2\alpha-\mu}\right).
$$

$(C4)$.~If~$N\geq7$ and $4<2\alpha+\mu<\frac{N+2}{2}$, then $u\in L^{p}(\mathbb{R}^{N})$ with
$$
p\in\left(\frac{2N}{N-2+2\alpha+\mu},\frac{2N}{2\alpha+\mu-4}\right).
$$
\end{thm}
We can improve part of the results to $L^{\infty}$ integrability in the following theorem.
\begin{thm}\label{ZX}
Assume that $N=3,4,5,6$, $0\leq\alpha<2$, $0<\mu<N$ and $N-2\leq2\alpha+\mu\leq\min\{N,4\}$. Let $u\in D^{1,2}(\mathbb{R}^N)$ be a positive solution of equation \eqref{Intcase}, then $u\in L^{\infty}(\mathbb{R}^{N})$.
\end{thm}
Furthermore, we will establish $C^{\infty}$ regularity of the solution away from the origin, that is
\begin{thm}\label{DF}
Assume that $N=3,4,5,6$, $0\leq\alpha<2$, $0<\mu<N$ and $N-2\leq2\alpha+\mu\leq\min\{N,4\}$. Let $u\in D^{1,2}(\mathbb{R}^N)$ be a positive solution of equation \eqref{Intcase}, then $u\in C^{\infty}(\mathbb{R}^{N}-\{0\})$.
\end{thm}

Next we are interested in the symmetry of the solutions of equation \eqref{Intcase}, we have the following results:

\begin{thm}\label{thm1.3}
Assume that $N\geq3$, $\alpha\geq0$, $0<\mu<N$, $0<2\alpha+\mu\leq\min\{N,4\}$.
Let $u\in  D^{1,2}(\mathbb{R}^N)$ be a positive solution of \eqref{Intcase},
then $u$ must be radially symmetric about origin.
\end{thm}

From Theorem \ref{thm1.3}, we know that the positive minimizer obtained in Theorem \ref{thm1.2} is radially symmetric about origin and decreasing. Thus we can characterize the asymptotic behavior at infinity.
\begin{thm}\label{thm1.4}
Assume that $N\geq3$, $\alpha\geq0$, $0<\mu<N$, $0<2\alpha+\mu\leq\min\{N,4\}$, then any nonnegative minimizer of $S_{\alpha, \mu}$ satisfies
\begin{equation}\label{1.6}
u(|x|)\leq\left[\frac{(N-\alpha)^{2}2^{\mu}}{\omega^{2}_{N-1}}\right]^{\frac{1}{2\cdot2^{\ast}_{\alpha,\mu}}}\left(\frac{1}{|x|}\right)^{\frac{N-2}{2}},~~x\neq0.
\end{equation}
Where $\omega_{N-1}$ is the area of the unit sphere in $\mathbb{R}^{N}$.
\end{thm}

The second part of this paper is to study the subcritical case
\begin{equation}\label{wSCC}
-\Delta u+u
=\frac{1}{|x|^{\alpha}}\left(\int_{\mathbb{R}^{N}}\frac{|u(y)|^{p}}
{|x-y|^{\mu}|y|^{\alpha}}dy\right)|u|^{p-2}u\hspace{4.14mm}\mbox{in}\hspace{1.14mm} \mathbb{R}^{N},
\end{equation}
where $N\geq3$, $0<\mu<N$, $\alpha\geq0$, $0<2\alpha+\mu< N$ and $2-\frac{2\alpha+\mu}{N}< p<\frac{2N-2\alpha-\mu}{N-2}$. Equation \eqref{wSCC} is closely related to the nonlocal Choquard equation. For $N=3$, $\alpha=\beta=0$, $p=2$ and $\mu=1$, equation \eqref{wSCC}
was introduced in mathematical physics by Pekar \cite{P1} to study the quantum theory of a polaron at rest. It was mentioned in \cite{LE2} that Choquard applied it as approximation to Hartree-Fock theory of one-component plasma. This equation was also proposed by Penrose in \cite{Pen} as a model of selfgravitating matter and was known as the Schr\"odinger-Newton equation. Mathematically, Lieb \cite{LE2} and
Lions \cite{Ls} studied the existence  and uniqueness of
positive solutions to equation \eqref{wSCC}. Moroz and Van
Schaftingen \cite{MS1} obtained the existence of radial symmetric ground state and they also considered in \cite{MS2} the existence of ground states under the assumption of Berestycki-Lions type. If the periodic potential $V(x)$ changes sign and $0$ lies in the gap of the spectrum of $-\Delta +V$, then the energy functional associated to the problem is strongly indefinite indeed. For this case, the existence of solution for $p=2$ was considered in \cite{BJS} and the authors developed reduction argument to obtain the existence of weak solution. Still for the strongly indefinite case, Ackermann \cite{AC} established the splitting lemma for the nonlocal nonlinearities and proved the existence of infinitely many geometrically distinct weak solutions. Alves and Yang investigated the quasilinear case in \cite{AY1, AY2}. The critical case in the sense of the classical Hardy-Littlewood-Sobolev inequality, Gao and Yang \cite{GY} studied the Brezis-Nirenberg problem on bounded domain. Alves et al. \cite{AGSY} considered the existence and concentration of the semiclassical solutions. For the strongly indefinite problem with critical term, Gao and Yang \cite{GY2} obtained the existence of ground state by generalized linking arguments.

Here we are interested in the subcritical weighted case, that is $\alpha=\beta\neq0$, we have the following existence result.
\begin{thm}\label{Existencesub}
 Assume that $N\geq3$, $\alpha\geq0$, $0<\mu<N$, $0<2\alpha+\mu\leq N$ and $2-\frac{2\alpha+\mu}{N}< p<\frac{2N-2\alpha-\mu}{N-2}$. Then there exists a ground state solution $u$ for \eqref{wSCC} in $H^{1}(\mathbb{R}^N)$.
\end{thm}

Notice that equation \eqref{wSCC} is non-periodic, so we can not use the Mountain-Pass Theorem and Lion's vanishing-nonvanishing arguments to obtain the existence of ground states solutions directly. In order to prove the existence of such solutions, we will use the rearrangement arguments and the Nehari manifold methods.

We can also establish a non-existence result for $p\geq\frac{2N-2\alpha-\mu}{N-2}$ or $p\leq\frac{2N-2\alpha-\mu}{N}$, which means the existence result in Theorem \eqref{Existencesub} is optimal. This result will be proved by establishing a Poho\v{z}aev identity.

\begin{thm}\label{Po}
Assume that $N\geq3$, $\alpha\geq0$, $0<\mu<N$, $0<\mu+2\alpha\leq N$ and $u\in W^{2,2}_{loc}(\mathbb{R}^{N})\cap L^{\frac{2Np}{2N-2\al-\mu}}_{loc}(\mathbb{R}^{N})$ is a solution of \eqref{wSCC}.
If $p\geq\frac{2N-2\alpha-\mu}{N-2}$ or $p\leq\frac{2N-2\alpha-\mu}{N}$, then $u\equiv0$.
\end{thm}

The qualitative properties of solutions of the Choquard equation had also attracted a lot of interests. The uniqueness and non-degeneracy of the ground states were proved by Lenzmann in \cite{Len}, Wei
and Winter in \cite{WW}.  Ma and Zhao \cite{MZ} classified the positive solutions of
\begin{eqnarray}\label{limit equation3}
-\Delta u+u=\Big(\int_{\R^N} \frac{|u(y)|^p}{|x-y|^\mu}dy\Big)|u|^{p-2}u,
\end{eqnarray}
and solved a longstanding open problem concerning the classification of all positive solutions to the nonlinear stationary Choquard equation. Moreover, they removed the restriction that those solutions minimize an
energy. Since the nonlinear term with a convolution is difficult to handle, the authors introduced an equivalent integral system. Namely, according to the properties of the Riesz and
the Bessel potentials, equation \eqref{limit equation3} is equivalent to
\begin{equation}\label{Intsubcri}
\begin{cases}
u(x)=\displaystyle\int_{\mathbb{R}^{N}}g_{2}(x-y)v(y)u(y)^{p-1}dy,\vspace{2mm}\\
v(x)=\displaystyle\int_{\mathbb{R}^{N}}\frac{u(y)^{p}}{|x-y|^{\mu}}dy,
\end{cases}
\end{equation}
where $g_{2}(x)$ is the Bessel kernel. Then they applied the method of moving planes in integral forms by Chen, Li and Ou \cite{CLO2} to obtain the radial symmetry of positive
solutions of \eqref{limit equation3}. Combining the Lieb's result in \cite{LE2}, they proved that such solutions are
unique. The regularity of solution of the Choquard equation with fractional operator was studied by Lei \cite{LY}. By using two regularity lifting lemmas introduced by Chen and Li
\cite{CL2}, the author was able to obtain the regularity for integrable solutions $u$ under some restrictions on the exponent $p$. In \cite{MS1}, Moroz and Van
Schaftingen obtained the existence of radial symmetric ground state and completely investigated the the regularity, decay behavior of solutions of \eqref{limit equation3}.

We will continue to study the qualitative properties of the subcritical weighted Choquard equation. Firstly we are going to investigate the regularity of the solutions of equation \eqref{wSCC}.
\begin{thm}\label{SRegularity}
 Assume that $N\geq3$, $\alpha\geq0$, $0<\mu<N$, $0<2\alpha+\mu\leq N$ and $2-\frac{2\alpha+\mu}{N}< p<\frac{2N-2\alpha-\mu}{N-2}$. Let $u\in H^{1}(\mathbb{R}^N)$ be a positive solution of equation \eqref{wSCC},
then $u\in C^\infty(\R^N-\{0\})$.
\end{thm}

As a generalization of the results obtained in \cite{MZ}, we have the symmetry property of the positive solutions for the weighted equation \eqref{wSCC}.
\begin{thm}\label{thm}
Assume $N\geq3$, $\alpha\geq0$, $0<\mu<N$, $2\alpha+\mu\leq3$ if $N=3$ while $2\alpha+\mu<4$ if $N\geq4$ and $2\leq p<\frac{2N-2\alpha-\mu}{N-2}$. Let $u\in H^{1}(\mathbb{R}^{N})$ be a positive solution of equation \eqref{wSCC},
then $u$ must be radially symmetric about origin.
\end{thm}

An outline of the paper is as follows. In section 2 we mainly focus on the critical case due to the weighted Hardy-Littlewood-Sobolev inequality. By establishing a nonlocal version of concentration-compactness principle, we are able to obtain the existence result. Translating the equation into an equivalent integral system, we apply a regularity lifting lemma by contracting mapping arguments to obtain the regularity of the solutions. In the final part of this section we will use the moving plane methods in integral form to study the symmetry of the positive solutions. In section 3 we will study the subcritical case. First we will establish the existence of ground states by Riesz rearrangement and Nehari manifold arguments. Then by establishing a Poho\v{z}aev identity we prove a non-existence result. In this part we will prove the regularity of the solutions by some iterative arguments and singular integral analysis. Finally we prove the symmetry of solutions by using the moving plane method in integral forms again.

\section{The critical case}
\subsection{Concentration-compactness principle}
To prove the  $(PS)$ condition, we prove a nonlocal version of the Br\'{e}zis-Lieb convergence lemma (see \cite{BL1}). The proof is analogous to that of Lemma 3.5 in \cite{AC} or Lemma 2.4 in \cite{MS1}, but we write it here for completeness. First, we recall that pointwise convergence of a bounded sequence implies weak convergence, see \cite[Proposition 5.4.7]{Wi2}.

\begin{lem}\label{BLN1} Let $N\geq3$, $q\in(1,+\infty)$ and $\{u_{n}\}$ is a bounded sequence in $L^{q}(\mathbb{R}^N)$. If $u_{n}\rightarrow u$ a.e. in $\mathbb{R}^N$ as $n\rightarrow\infty$, then $u_{n}\rightharpoonup u$ weakly in $L^{q}(\mathbb{R}^N)$.
\end{lem}

\begin{lem} \label{BLN}Let $N\geq3$, $\alpha\geq0$, $0<\mu<N$, $2\alpha+\mu\leq N$ and $1\leq p\leq\frac{2N-2\alpha-\mu}{N-2}$. If $\{u_{n}\}$ is a bounded sequence in $L^{\frac{2Np}{2N-2\alpha-\mu}}(\mathbb{R}^N)$ such that $u_{n}\rightarrow u$ a.e. in $\mathbb{R}^N$ as $n\rightarrow\infty$, then the following hold:
$$
\int_{\mathbb{R}^N}
\int_{\mathbb{R}^N}\frac{|u_{n}(x)|^{p}|u_{n}(y)|^{p}}
{|x|^{\alpha}|x-y|^{\mu}|y|^{\alpha}}dxdy
-\int_{\mathbb{R}^N}
\int_{\mathbb{R}^N}\frac{|(u_{n}-u)(x)|^{p}|(u_{n}-u)(y)|^{p}}
{|x|^{\alpha}|x-y|^{\mu}|y|^{\alpha}}dxdy
$$
$$
\rightarrow\int_{\mathbb{R}^N}
\int_{\mathbb{R}^N}\frac{|u(x)|^{p}|u(y)|^{p}}
{|x|^{\alpha}|x-y|^{\mu}|y|^{\alpha}}dxdy
$$
as $n\rightarrow\infty$.
\end{lem}
\begin{proof}
Since $\{u_{n}\}$ is a bounded sequence in $L^{\frac{2Np}{2N-2\alpha-\mu}}(\mathbb{R}^N)$ such that $u_{n}\rightarrow u$ almost everywhere in $\mathbb{R}^N$ as $n\rightarrow\infty$, then similar to the proof of the Br\'{e}zis-Lieb Lemma \cite{BL1}, we know that
\begin{equation}\label{b1}
|u_{n}-u|^{p}-|u_{n}|^{p}\rightarrow|u|^{p}
\end{equation}
in  $L^{\frac{2N}{2N-2\alpha-\mu}}(\mathbb{R}^N)$ as $n\rightarrow\infty$.
Let $$Tf(x)=\displaystyle\int_{\mathbb{R}^{N}}\frac{f(y)}
{|x|^{\alpha}|x-y|^{\mu}|y|^{\beta}}dy,$$ the weighted Hardy-Littlewood-Sobolev inequality can be written in the form of
$$
|Tf|_{l}=\sup_{|h|_{s}=1}\langle Tf,h\rangle\leq C|f|_{r},
$$
where $\frac{1}{r}+\frac{\alpha+\beta+\mu}{N}=1+\frac{1}{l}$, $\frac{1}{l}+\frac{1}{s}=1$.
It implies that
\begin{equation}\label{b2}
\displaystyle\int_{\mathbb{R}^{N}}\frac{|u_{n}|^{p}}
{|x|^{\alpha}|x-y|^{\mu}|y|^{\alpha}}dy-\displaystyle\int_{\mathbb{R}^{N}}\frac{|u_{n}-u|^{p}}
{|x|^{\alpha}|x-y|^{\mu}|y|^{\alpha}}dy  \rightarrow\displaystyle\int_{\mathbb{R}^{N}}\frac{|u|^{p}}
{|x|^{\alpha}|x-y|^{\mu}|y|^{\alpha}}dy
\end{equation}
in $L^{\frac{2N}{2\alpha+\mu}}(\mathbb{R}^N)$ as $n\rightarrow\infty$. On the other hand, we notice that
\begin{equation}\label{b3}
\aligned
&\int_{\mathbb{R}^N}\left(|x|^{-\mu}\ast \left(\frac{1}{|x|^{\alpha}}|u_{n}|^{2_{\alpha, \mu}^{\ast}}\right)\right)\frac{1}{|x|^{\alpha}}|u_{n}|^{2_{\alpha, \mu}^{\ast}}dx
-\int_{\mathbb{R}^N}\left(|x|^{-\mu}\ast \left(\frac{1}{|x|^{\alpha}}|u_{n}-u|^{2_{\alpha, \mu}^{\ast}}\right)\right)\frac{1}{|x|^{\alpha}}|u_{n}-u|^{2_{\alpha, \mu}^{\ast}}dx\\
&=\int_{\mathbb{R}^N}\left(|x|^{-\mu}\ast \left(\frac{1}{|x|^{\alpha}}|u_{n}|^{2_{\alpha, \mu}^{\ast}}-\frac{1}{|x|^{\alpha}}|u_{n}-u|^{2_{\alpha, \mu}^{\ast}}\right)\right)\frac{1}{|x|^{\alpha}}
(|u_{n}|^{2_{\alpha, \mu}^{\ast}}-|u_{n}-u|^{2_{\alpha, \mu}^{\ast}})dx\\
&\hspace{4mm}+2\int_{\mathbb{R}^N}\left(|x|^{-\mu}\ast \left(\frac{1}{|x|^{\alpha}}|u_{n}|^{2_{\alpha, \mu}^{\ast}}-\frac{1}{|x|^{\alpha}}|u_{n}-u|^{2_{\alpha, \mu}^{\ast}}\right)\right)
\frac{1}{|x|^{\alpha}}|u_{n}-u|^{2_{\alpha, \mu}^{\ast}}dx.\\
\endaligned
\end{equation}
By Lemma \ref{BLN1}, we have that
\begin{equation}\label{b4}
|u_{n}-u|^{2_{\alpha, \mu}^{\ast}}\rightharpoonup0
\end{equation}
in $L^{\frac{2N}{2N-2\alpha-\mu}}(\mathbb{R}^N)$ as $n\rightarrow\infty$. From \eqref{b1}-\eqref{b4}, we know that the result holds.
\end{proof}

To describe the lack of compactness of the injection from $D^{1,2}(\mathbb{R}^N)$ to $L^{2^*}(\mathbb{R}^N)$, P.L. Lions established the well known Concentration-compactness principles  \cite{Ls1, Ls2, Ls3, Ls4}. Here we would like to recall the second concentration-compactness principle \cite{Ls1} for the convenience of the readers.
\begin{lem}\label{Concentration-compactness principle}
Let $\{u_{n}\}$ be a bounded sequence in $D^{1,2}(\mathbb{R}^N)$ converging weakly and a.e. to some $u_0\in D^{1,2}(\mathbb{R}^N)$. $|\nabla u_{n}|^{2}\rightharpoonup \omega$, $|u_{n}|^{2^*}\rightharpoonup \zeta$ weakly in the sense of measures where $\omega$ and $\zeta$ are bounded non-negative measures on $\mathbb{R}^N$. Then we have:\\
(1). there exists some at most countable set $I$, a family $\{z_i:i\in I\}$ of distinct points in $\mathbb{R}^N$, and a family $\{\zeta_i:i\in I\}$ of positive numbers such that
$$
\zeta=|u_0|^{2^*}+\sum_{i\in I}\zeta_i\delta_{z_i},
$$
where $\delta_{x}$ is the Dirac-mass of mass 1 concentrated at $x\in\mathbb{R}^N$.\\
(2). In addition we have
$$
\omega\geq|\nabla u_0|^{2}+\sum_{i\in I}\omega_i\delta_{z_i}
$$
for some family $\{\omega_i:i\in I\}$, $\omega_i>0$ satisfying
$$
S\zeta_i^{\frac{2}{2^*}}\leq\omega_i,\ \ \mbox{for all } i\in I.
$$
In particular, $\sum_{i\in I}\zeta_i^{\frac{2}{2^*}}<\infty$.
\end{lem}

 However, the second concentration-compactness principle, roughly
speaking, is only concerned with a possible concentration of a weakly convergent sequence at finite points and it does not provide any information about the loss
of mass of a sequence at infinity. The following concentration-compactness principle at infinity was developed by Chabrowski \cite{C} which provided some quantitative information about the loss of mass of a sequence at infinity.

\begin{lem}\label{Concentration-compactness principle2}
Let $\{u_{n}\}\subset D^{1,2}(\mathbb{R}^N)$ be a sequence in Lemma \ref{Concentration-compactness principle} and define
$$
\omega_{\infty}:=\lim_{R\rightarrow\infty}\overline{\lim}_{n\rightarrow\infty}\int_{|x|\geq R}|\nabla u_{n}|^{2}dx,\ \ \
\zeta_{\infty}:=\lim_{R\rightarrow\infty}\overline{\lim}_{n\rightarrow\infty}\int_{|x|\geq R}| u_{n}|^{2^{\ast}}dx.
$$
Then it follows that
$$
S\zeta_{\infty}^{\frac{2}{2^{\ast}}}\leq \omega_{\infty},
$$
$$
\overline{\lim}_{n\rightarrow\infty}|\nabla u_{n}|_{2}^{2}=\int_{\mathbb{R}^N}d\omega+\omega_{\infty},
$$
$$
\overline{\lim}_{n\rightarrow\infty}|u_{n}|_{2^{\ast}}^{2^{\ast}}=\int_{\mathbb{R}^N}d\zeta+\zeta_{\infty}.
$$
\end{lem}

 The concentration-compactness principles \cite{Ls1, Ls2, Ls3, Ls4}
 help not only to investigate
 the behaviour of the weakly convergent sequences in Sobolev spaces where
the lack of compactness occurs either due to the appearance of a critical Sobolev
exponent or due to the unboundedness of a domain and but also to find level sets of a given variational functional for which the
Palais-Smale condition holds.

\begin{lem}\label{CCP1} Let $N\geq3$, $\alpha\geq0$, $0<\mu<N$, $2\alpha+\mu\leq N$, $\{u_{n}\}$ be a bounded sequence in $D^{1,2}(\mathbb{R}^N)$ converging weakly and a.e. to some $u_{0}$ and $\omega,  \omega_{\infty}, \zeta, \zeta_{\infty}$ be the bounded nonnegative measures in Lemma \ref{Concentration-compactness principle}  and Lemma \ref{Concentration-compactness principle2}. Assume that
$$
\Big(\int_{\mathbb{R}^N}\frac{
|u_{n}(y)|^{2_{\alpha, \mu}^{\ast}}}{|x-y|^{\mu}|y|^{\alpha}}dy\Big)
\frac{|u_{n}(x)|^{2_{\alpha, \mu}^{\ast}}}{|x|^{\alpha}}
\rightharpoonup \nu$$
weakly in the sense of measure where $\nu$ is a bounded positive measure on $\mathbb{R}^N$ and define
$$\aligned
\nu_{\infty}&:=\lim_{R\rightarrow\infty}\overline{\lim}_{n\rightarrow\infty}\int_{|x|\geq R}\Big(\int_{\mathbb{R}^N}\frac{
|u_{n}(y)|^{2_{\alpha, \mu}^{\ast}}}{|x-y|^{\mu}|y|^{\alpha}}dy\Big)
\frac{|u_{n}(x)|^{2_{\alpha, \mu}^{\ast}}}{|x|^{\alpha}}dx.
\endaligned$$
Then there exists a countable sequence of points $\{z_{i}\}_{i\in I}\subset \mathbb{R}^N $ and families  of positive numbers $\{\nu_i:i\in I\}$, $\{\zeta_i:i\in I\}$ and $\{\omega_i:i\in I\}$ such that
\begin{equation}\label{cp5}
\nu=\Big(\int_{\mathbb{R}^N}\frac{
|u_{0}(y)|^{2_{\alpha, \mu}^{\ast}}}{|x-y|^{\mu}|y|^{\alpha}}dy\Big)
\frac{|u_{0}(x)|^{2_{\alpha, \mu}^{\ast}}}{|x|^{\alpha}}+ \Sigma_{i\in I}\nu_{i}\delta_{z_{i}},\ \  \Sigma_{i\in I}\nu_{i}^{\frac{1}{2_{\alpha, \mu}^{\ast}}}<\infty,
\end{equation}
\begin{equation}\label{cp51}
\omega\geq|\nabla u_0|^{2}+\sum_{i\in I}\omega_i\delta_{z_i},
\end{equation}
\begin{equation}\label{cp52}
\zeta\geq|u_0|^{2^*}+\sum_{i\in I}\zeta_i\delta_{z_i},
\end{equation}
and
\begin{equation}\label{cp6}
 \ S_{\alpha, \mu}\nu_{i}^{\frac{1}{2_{\alpha, \mu}^{\ast}}}\leq\omega_{i},\ \  \nu_{i}^{\frac{N}{2N-2\alpha-\mu}}\leq C(N,\mu,\alpha)^{\frac{N}{2N-2\alpha-\mu}}\zeta_{i},
\end{equation}
where $\delta_{x}$ is the Dirac-mass of mass 1 concentrated at $x\in\mathbb{R}^N$.

For the energy at infinity, we have
\begin{equation}\label{cp51}
\overline{\lim}_{n\rightarrow\infty}\int_{\mathbb{R}^N}\int_{\mathbb{R}^N}
\frac{|u_{n}(x)|^{2_{\alpha, \mu}^{\ast}}
|u_{n}(y)|^{2_{\alpha, \mu}^{\ast}}}{|x|^{\alpha}|x-y|^{\mu}|y|^{\alpha}}dxdy
=\nu_{\infty}+\int_{\mathbb{R}^N}d\nu,
\end{equation}
and
\begin{equation}\label{cp71}
C(N,\mu,\alpha)^{\frac{-2N}{2N-2\alpha-\mu}}\nu_{\infty}^{\frac{2N}{2N-2\alpha-\mu}}\leq \zeta_{\infty}\left(\int_{\mathbb{R}^N}d\zeta+\zeta_{\infty}\right),\  \  S_{\alpha, \mu}^{2}\nu_{\infty}^{\frac{2}{2_{\alpha, \mu}^{\ast}}}\leq \omega_{\infty}\left(\int_{\mathbb{R}^N}d\omega+\omega_{\infty}\right).
\end{equation}
Moreover, if $u_0=0$ and $\displaystyle\int_{\mathbb{R}^N}d\omega
=S_{\alpha, \mu}\left(\int_{\mathbb{R}^N}d\nu\right)^{\frac{1}{2_{\alpha, \mu}^{\ast}}}$, then $\nu$ is concentrated at a single point.
\end{lem}
\begin{proof}
Since $\{u_{n}\}$ is a bounded sequence in $D^{1,2}(\mathbb{R}^N)$ converging weakly to $u_0$, denote by $v_{n}:=u_{n}-u_0$, we have $v_{n}(x)\rightarrow0$ a.e. in $\mathbb{R}^N$  and $v_{n}$ converges weakly to $0$ in $D^{1,2}(\R^N)$. Applying Lemma \ref{BLN}, in the sense of measure, we have
\begin{center}
$|\nabla v_{n}|^{2}\rightharpoonup\varpi:=\omega-|\nabla u_0|^{2}$,
\end{center}
\begin{center}
$\Big(\displaystyle\int_{\mathbb{R}^N}\frac{
|v_{n}(y)|^{2_{\alpha, \mu}^{\ast}}}{|x-y|^{\mu}|y|^{\alpha}}dy\Big)
\frac{|v_{n}(x)|^{2_{\alpha, \mu}^{\ast}}}{|x|^{\alpha}}\rightharpoonup\kappa:
=\nu-\Big(\displaystyle\int_{\mathbb{R}^N}\frac{
|u_{0}(y)|^{2_{\alpha, \mu}^{\ast}}}{|x-y|^{\mu}|y|^{\alpha}}dy\Big)
\frac{|u_{0}(x)|^{2_{\alpha, \mu}^{\ast}}}{|x|^{\alpha}}$,
\end{center}
\begin{center}
$|v_{n}|^{2^{\ast}}\rightharpoonup\varsigma:=\zeta-|u_0|^{2^{\ast}}$.
\end{center}

To prove the possible concentration at finite points, we first claim that
\begin{equation}\label{cp0}
\Big|\int_{\mathbb{R}^N}\Big(\int_{\mathbb{R}^N}\frac{
|\phi v_{n}(y)|^{2_{\alpha, \mu}^{\ast}}}{|x-y|^{\mu}|y|^{\alpha}}dy\Big)
\frac{|\phi v_{n}(x)|^{2_{\alpha, \mu}^{\ast}}}{|x|^{\alpha}}dx-\int_{\mathbb{R}^N}\Big(\int_{\mathbb{R}^N}\frac{
| v_{n}(y)|^{2_{\alpha, \mu}^{\ast}}}{|x-y|^{\mu}|y|^{\alpha}}dy\Big)
\frac{|\phi (x)|^{2_{\alpha, \mu}^{\ast}}|\phi v_{n}(x)|^{2_{\alpha, \mu}^{\ast}}}{|x|^{\alpha}}dx\Big|\rightarrow0,
\end{equation}
where $\phi\in \mathcal{C}_{0}^{\infty}(\mathbb{R}^N)$.

In fact, we denote
$$
\Phi_{n}(x):=\Big(\int_{\mathbb{R}^N}\frac{
|\phi v_{n}(y)|^{2_{\alpha, \mu}^{\ast}}}{|x-y|^{\mu}|y|^{\alpha}}dy\Big)
\frac{|\phi v_{n}(x)|^{2_{\alpha, \mu}^{\ast}}}{|x|^{\alpha}}-\Big(\int_{\mathbb{R}^N}\frac{
| v_{n}(y)|^{2_{\alpha, \mu}^{\ast}}}{|x-y|^{\mu}|y|^{\alpha}}dy\Big)
\frac{|\phi (x)|^{2_{\alpha, \mu}^{\ast}}|\phi v_{n}(x)|^{2_{\alpha, \mu}^{\ast}}}{|x|^{\alpha}}.
$$
Since $\phi\in \mathcal{C}_{0}^{\infty}(\mathbb{R}^N)$, we have for every $\delta>0$ there exists $M>0$ such that
\begin{equation}\label{cp00}
\int_{|x|\geq M}|\Phi_{n}(x)|dx<\delta \ \ \ (\forall n\geq1).
\end{equation}
Since the Riesz potential defines a linear operator, from the fact that $v_n(x)\to 0$ a.e. in $\mathbb{R}^N$ we know that
$$
\int_{\mathbb{R}^N}\frac{
| v_{n}(y)|^{2_{\alpha, \mu}^{\ast}}}{|x-y|^{\mu}|y|^{\alpha}}dy\rightarrow 0
$$
a.e. in $\mathbb{R}^N$ and so we have $\Phi_{n}(x)\rightarrow0$ a.e. in $\mathbb{R}^N$. Notice that
$$
\aligned
\Phi_{n}(x)&=\int_{\mathbb{R}^N}
\frac{(|\phi (y)|^{2_{\alpha, \mu}^{\ast}}-|\phi (x)|^{2_{\alpha, \mu}^{\ast}})|v_{n}(y)|^{2_{\alpha, \mu}^{\ast}}}{|x-y|^{\mu}|y|^{\alpha}}dy\frac{|\phi v_{n}(x)|^{2_{\alpha, \mu}^{\ast}}}{|x|^{\alpha}}\\
:&=\int_{\mathbb{R}^N}
L(x,y)|v_{n}(y)|^{2_{\alpha, \mu}^{\ast}}dy\frac{|\phi v_{n}(x)|^{2_{\alpha, \mu}^{\ast}}}{|x|^{\alpha}}.
\endaligned
$$
For almost all $x$, there exists $R>0$ large enough such that
$$
\int_{\mathbb{R}^N}
L(x,y)|v_{n}(y)|^{2_{\alpha, \mu}^{\ast}}dy
=\int_{|y|\leq R}
L(x,y)|v_{n}(y)|^{2_{\alpha, \mu}^{\ast}}dy
-|\phi (x)|^{2_{\alpha, \mu}^{\ast}}\int_{|y|\geq R}
\frac{|v_{n}(y)|^{2_{\alpha, \mu}^{\ast}}}{|x-y|^{\mu}|y|^{\alpha}}dy.
$$
 As observed in \cite{Ls2} that $L(x,y)\in L^{r}(B_{R})$ for each $x$, where $r<\frac{N}{\mu+\alpha-1}$ if $\mu>1$, $r<\frac{N}{\alpha}$ if $\mu\leq1$. By the Young inequality, there exists $s>\frac{2N}{\mu}$ such that $$
\Big(\int_{B_{M}}\Big(\int_{B_{R}}
L(x,y)|v_{n}(y)|^{2_{\alpha, \mu}^{\ast}}dy\Big)^{s}dx\Big)^{\frac{1}{s}}\leq C_{\phi}|L(x,y)|_{r}||v_{n}|^{2_{\alpha, \mu}^{\ast}}|_{\frac{2N}{2N-2\alpha-\mu}}\leq C_{\phi}',
$$
where $M$ is given in \eqref{cp00}. It is easy to see that for $R>0$ large enough
$$
\Big(\int_{B_{M}}
\Big(|\phi (x)|^{2_{\alpha, \mu}^{\ast}}\int_{|y|\geq R}
\frac{|v_{n}(y)|^{2_{\alpha, \mu}^{\ast}}}{|x-y|^{\mu}|y|^{\alpha}}dy\Big)^{s}dx\Big)^{\frac{1}{s}}\leq C
$$
and so we have
$$
\Big(\int_{B_{M}}\Big(\int_{\mathbb{R}^N}
L(x,y)|v_{n}(y)|^{2_{\alpha, \mu}^{\ast}}dy\Big)^{s}dx\Big)^{\frac{1}{s}}\leq C_{\phi}''.
$$
Denote $\tau:=\frac{1}{2}\cdot\frac{s\mu-2N}{2N+2Ns-s\mu}$, then we can get
$$
\int_{B_{M}}|\Phi_{n}(x)|^{1+\tau}dx\leq \Big(\int_{B_{M}}\Big(\int_{\mathbb{R}^N}
L(x,y)|v_{n}(y)|^{2_{\alpha, \mu}^{\ast}}dy\Big)^{s}dx\Big)^{\frac{1+\tau}{s}}
\Big(\int_{B_{M}}|\phi v_{n}|^{2^{\ast}}dx\Big)^{\frac{(2N-2\alpha-\mu)(1+\tau)}{2N}}
$$
$$
\Big(\int_{B_{M}}\frac{1}{|x|^{\frac{2Ns\alpha(1+\tau)}{2Ns-2N(1+\tau)-s(2N-2\alpha-\mu)(1+\tau)}}}dx
\Big)^{\frac{2Ns-2N(1+\tau)-s(2N-2\alpha-\mu)(1+\tau)}{2Ns}}\leq C_{\phi}'',
$$
thanks to $\frac{2Ns\alpha(1+\tau)}{2Ns-2N(1+\tau)-s(2N-2\alpha-\mu)(1+\tau)}<N$. Combining this and $\Phi_{n}(x)\rightarrow0$ a.e. in $\mathbb{R}^N$, we can get
$$
\int_{B_{M}}|\Phi_{n}(x)|dx\rightarrow0 \ \ \ (n\rightarrow\infty).
$$
By this fact and \eqref{cp00},  we have
$$
\int_{\mathbb{R}^N}|\Phi_{n}(x)|dx\rightarrow0,
$$
the claim \eqref{cp0} is thus proved.

Now for all $\phi\in \mathcal{C}_{0}^{\infty}(\mathbb{R}^N)$, by the weighted Hardy-Littlewood-Sobolev inequality, we have
$$
\int_{\mathbb{R}^N}\Big(\int_{\mathbb{R}^N}\frac{
|\phi v_{n}(y)|^{2_{\alpha, \mu}^{\ast}}}{|x-y|^{\mu}|y|^{\alpha}}dy\Big)
\frac{|\phi v_{n}(x)|^{2_{\alpha, \mu}^{\ast}}}{|x|^{\alpha}}dx
\leq C(N,\mu,\alpha)|\phi v_{n}|_{2^{\ast}}^{2\cdot2_{\alpha, \mu}^{\ast}}.
$$
And so \eqref{cp0} implies that
$$
\int_{\mathbb{R}^N}|\phi(x)|^{2\cdot2_{\alpha, \mu}^{\ast}}\Big(\int_{\mathbb{R}^N}\frac{
| v_{n}(y)|^{2_{\alpha, \mu}^{\ast}}}{|x-y|^{\mu}|y|^{\alpha}}dy\Big)
\frac{| v_{n}(x)|^{2_{\alpha, \mu}^{\ast}}}{|x|^{\alpha}}dx
\leq C(N,\mu,\alpha)|\phi v_{n}|_{2^{\ast}}^{2\cdot2_{\alpha, \mu}^{\ast}}+o(1).
$$
Passing to the limit as $n\rightarrow+\infty$ we obtain
\begin{equation}\label{cp9}
\int_{\mathbb{R}^N}|\phi(x)|^{2\cdot2_{\alpha, \mu}^{\ast}}d\kappa
\leq C(N,\mu,\alpha)\Big(\int_{\mathbb{R}^N}|\phi|^{2^{\ast}}d\varsigma\Big)^{\frac{2N-2\alpha-\mu}{N}}.
\end{equation}
Applying Lemma 1.2 in \cite{Ls1} we know \eqref{cp52} holds.

Taking $\phi=\chi_{\{z_{i}\}}$, $i\in I$, in \eqref{cp9}, we get
$$
\nu_{i}^{\frac{N}{2N-2\alpha-\mu}}\leq C(N,\mu,\alpha)^{\frac{N}{2N-2\alpha-\mu}}\zeta_{i}, \ \forall i \in I.
$$
By the definition of $S_{\alpha, \mu}$, we also have
$$
\left(\int_{\mathbb{R}^N}\Big(\int_{\mathbb{R}^N}\frac{
|\phi v_{n}(y)|^{2_{\alpha, \mu}^{\ast}}}{|x-y|^{\mu}|y|^{\alpha}}dy\Big)
\frac{|\phi v_{n}(x)|^{2_{\alpha, \mu}^{\ast}}}{|x|^{\alpha}}dx\right)^{\frac{N-2}{2N-2\alpha-\mu}}
S_{\alpha, \mu}\leq\displaystyle\int_{\mathbb{R}^N}|\nabla (\phi v_{n})|^{2}dx.
$$
By \eqref{cp0} and $v_{n}\rightarrow0$ in $L_{loc}^{2}(\mathbb{R}^N)$, we have
$$
\left(\int_{\mathbb{R}^N}|\phi(x)|^{2\cdot2_{\alpha, \mu}^{\ast}}\Big(\int_{\mathbb{R}^N}\frac{
| v_{n}(y)|^{2_{\alpha, \mu}^{\ast}}}{|x-y|^{\mu}|y|^{\alpha}}dy\Big)
\frac{| v_{n}(x)|^{2_{\alpha, \mu}^{\ast}}}{|x|^{\alpha}}dx\right)^{\frac{N-2}{2N-2\alpha-\mu}}
S_{\alpha, \mu}\leq\displaystyle\int_{\mathbb{R}^N}\phi^{2}|\nabla v_{n}|^{2}dx+o(1).
$$
Passing to the limit as $n\rightarrow+\infty$ we obtain
\begin{equation}\label{cp8}
\left(\int_{\mathbb{R}^N}|\phi(x)|^{2\cdot2_{\alpha, \mu}^{\ast}}d\kappa\right)^{\frac{N-2}{2N-2\alpha-\mu}}
S_{\alpha, \mu}\leq\displaystyle\int_{\mathbb{R}^N}\phi^{2}d\varpi.
\end{equation}
Applying Lemma 1.2 in \cite{Ls1} again we know \eqref{cp51} holds.
Now by taking $\phi=\chi_{\{z_{i}\}}$, $i\in I$, in \eqref{cp8}, we get
$$
S_{\alpha, \mu}\nu_{i}^{\frac{1}{2_{\alpha, \mu}^{\ast}}}\leq\omega_{i}, \ \forall i \in I.
$$
Thus we proved \eqref{cp5} and \eqref{cp6}.

Next we are going to prove the possible loss of mass at infinity. For $R>1$, let $\psi_{R}\in \mathcal{C}^{\infty}(\mathbb{R}^N)$ be such that $\psi_{R}(x)=1$ for $|x|>R+1$, $\psi_{R}(x)=0$ for $|x|<R$ and $0\leq\psi_{R}(x)\leq1$ on $\mathbb{R}^N$. For every $R>1$, we have

$$\aligned
&\overline{\lim}_{n\rightarrow\infty}\int_{\mathbb{R}^N}\int_{\mathbb{R}^N}
\frac{|u_{n}(y)|^{2_{\alpha, \mu}^{\ast}}|u_{n}(x)|^{2_{\alpha, \mu}^{\ast}}}{|x|^{\alpha}|x-y|^{\mu}|y|^{\alpha}}
dydx\\
&=\overline{\lim}_{n\rightarrow\infty}\Big(\int_{\mathbb{R}^N}\int_{\mathbb{R}^N}
\frac{|u_{n}(y)|^{2_{\alpha, \mu}^{\ast}}|u_{n}(x)|^{2_{\alpha, \mu}^{\ast}}\psi_{R}(x)}
{|x|^{\alpha}|x-y|^{\mu}|y|^{\alpha}}
dydx+\int_{\mathbb{R}^N}\int_{\mathbb{R}^N}
\frac{|u_{n}(y)|^{2_{\alpha, \mu}^{\ast}}|u_{n}(x)|^{2_{\alpha, \mu}^{\ast}}(1-\psi_{R}(x))}
{|x|^{\alpha}|x-y|^{\mu}|y|^{\alpha}}
dydx\Big)\\
&=\overline{\lim}_{n\rightarrow\infty}\int_{\mathbb{R}^N}\int_{\mathbb{R}^N}
\frac{|u_{n}(y)|^{2_{\alpha, \mu}^{\ast}}|u_{n}(x)|^{2_{\alpha, \mu}^{\ast}}\psi_{R}(x)}{|x|^{\alpha}|x-y|^{\mu}|y|^{\alpha}}
dydx+\int_{\mathbb{R}^N}(1-\psi_{R})d\nu.
\endaligned$$
When $R\rightarrow\infty$, we obtain, by Lebesgue's theorem,
$$
\overline{\lim}_{n\rightarrow\infty}\int_{\mathbb{R}^N}\int_{\mathbb{R}^N}
\frac{|u_{n}(y)|^{2_{\alpha, \mu}^{\ast}}|u_{n}(x)|^{2_{\alpha, \mu}^{\ast}}}{|x|^{\alpha}|x-y|^{\mu}|y|^{\alpha}}
dydx=\nu_{\infty}+\int_{\mathbb{R}^N}d\nu.
$$
By the weighted Hardy-Littlewood-Sobolev inequality, we have
$$\aligned
\nu_{\infty}&=\lim_{R\rightarrow\infty}
\overline{\lim}_{n\rightarrow\infty}
\int_{\mathbb{R}^N}
\Big(\int_{\mathbb{R}^N}
\frac{|u_{n}(y)|^{2_{\alpha, \mu}^{\ast}}}{|x-y|^{\mu}|y|^{\alpha}}dy\Big)\frac{|\psi_{R} u_{n}(x)|^{2_{\alpha, \mu}^{\ast}}}{|x|^{\alpha}}dx\\
&\leq C(N,\mu,\alpha)\lim_{R\rightarrow\infty}\overline{\lim}_{n\rightarrow\infty}
\Big(\int_{\mathbb{R}^N}|u_{n}|^{2^{\ast}}dx
\int_{\mathbb{R}^N}|\psi_{R}u_{n}|^{2^{\ast}}dx\Big)^{\frac{2N-2\alpha-\mu}{2N}}\\
&= C(N,\mu,\alpha)\left(\zeta_{\infty}\left(\int_{\mathbb{R}^N}d\zeta+\zeta_{\infty}\right)\right)^{\frac{2N-2\alpha-\mu}{2N}},
\endaligned$$
which means
$$
C(N,\mu,\alpha)^{\frac{-2N}{2N-2\alpha-\mu}}\nu_{\infty}^{\frac{2N}{2N-2\alpha-\mu}}\leq \zeta_{\infty}\left(\int_{\mathbb{R}^N}d\zeta+\zeta_{\infty}\right).
$$
Similarly, by the definition of $S_{\alpha, \mu}$ and $\nu_{\infty}$, we have
$$
\aligned
\nu_{\infty}&=\lim_{R\rightarrow\infty}
\overline{\lim}_{n\rightarrow\infty}
\int_{\mathbb{R}^N}
\Big(\int_{\mathbb{R}^N}
\frac{|u_{n}(y)|^{2_{\alpha, \mu}^{\ast}}}{|x-y|^{\mu}|y|^{\alpha}}dy\Big)\frac{|\psi_{R} u_{n}(x)|^{2_{\alpha, \mu}^{\ast}}}{|x|^{\alpha}}dx\\
&\leq C(N,\mu,\alpha)\lim_{R\rightarrow\infty}\overline{\lim}_{n\rightarrow\infty}
\Big(\int_{\mathbb{R}^N}|u_{n}|^{2^{\ast}}dx
\int_{\mathbb{R}^N}|\psi_{R}u_{n}|^{2^{\ast}}dx\Big)^{\frac{2N-2\alpha-\mu}{2N}}\\
&\leq S_{\alpha, \mu}^{-2_{\alpha, \mu}^{\ast}}\lim_{R\rightarrow\infty}\overline{\lim}_{n\rightarrow\infty}
\Big(\int_{\mathbb{R}^N}|\nabla u_{n}|^{2}dx
\int_{\mathbb{R}^N}|\nabla(\psi_{R}u_{n})|^{2}dx\Big)^{\frac{2_{\alpha, \mu}^{\ast}}{2}}\\
&= S_{\alpha, \mu}^{-2_{\alpha, \mu}^{\ast}}\left(\omega_{\infty}\left(\int_{\mathbb{R}^N}d\omega
+\omega_{\infty}\right)\right)^{\frac{2_{\alpha, \mu}^{\ast}}{2}},
\endaligned$$
which means
$$
 S_{\alpha, \mu}^{2}\nu_{\infty}^{\frac{2}{2_{\alpha, \mu}^{\ast}}}\leq \omega_{\infty}\left(\int_{\mathbb{R}^N}d\omega+\omega_{\infty}\right).
$$

If $u_0=0$ then $\kappa=\nu$ and $\varpi=\omega$. The H\"{o}lder inequality and \eqref{cp8} imply that, for $\phi\in \mathcal{C}_{0}^{\infty}(\mathbb{R}^N)$,
$$
\left(\int_{\mathbb{R}^N}|\phi(x)|^{2\cdot2_{\alpha, \mu}^{\ast}}d\nu\right)^{\frac{N-2}{2N-2\alpha-\mu}}
S_{\alpha, \mu}\leq\left(\int_{\mathbb{R}^N}d\omega\right)^{\frac{N-2\alpha-\mu+2}{2N-2\alpha-\mu}}
\left(\int_{\mathbb{R}^N}\phi^{2\cdot2_{\alpha, \mu}^{\ast}}d\omega\right)^{\frac{N-2}{2N-2\alpha-\mu}}.
$$
Thus we can deduce that $\nu=S_{\alpha, \mu}^{-2_{\alpha, \mu}^{\ast}}
\left(\displaystyle\int_{\mathbb{R}^N}d\omega\right)^{\frac{N-2\alpha-\mu+2}{N-2}}\omega$. It follows from \eqref{cp8} that, for $\phi\in \mathcal{C}_{0}^{\infty}(\mathbb{R}^N)$,
$$
\left(\int_{\mathbb{R}^N}|\phi(x)|^{2\cdot2_{\alpha, \mu}^{\ast}}d\nu\right)^{\frac{N-2}{2N-2\alpha-\mu}}
\left(\int_{\mathbb{R}^N}d\nu\right)^{\frac{N-2\alpha-\mu+2}{2N-2\alpha-\mu}}\leq
\int_{\mathbb{R}^N}|\phi|^{2}d\nu.
$$
And then for each open set $\Omega$,
$$
\nu(\Omega)^{\frac{N-2}{2N-2\alpha-\mu}}\nu(\mathbb{R}^N)^{\frac{N-2\alpha-\mu+2}{2N-2\alpha-\mu}}
\leq\nu(\Omega).
$$
It follows that $\nu$ is concentrated at a single point.
\end{proof}

\subsection{Existence of Extremal functions} In this subsection we are going to prove the existence of extremal functions of the minimizing problem by the second concentration-compactness lemma established in Subsection 2.1.

{\bf{Proof of Theorem \ref{thm1.2}}. }
Define the L\'{e}vy concentration functions
$$
Q_{n}(\tau):=\sup_{z\in\mathbb{R}^N}\int_{B(z,\tau)}\Big(\int_{\mathbb{R}^N}\frac{
|u_{n}(y)|^{2_{\alpha, \mu}^{\ast}}}{|x-y|^{\mu}|y|^{\alpha}}dy\Big)
\frac{|u_{n}(x)|^{2_{\alpha, \mu}^{\ast}}}{|x|^{\alpha}}dx.
$$
Since for every $n$,
$$
\lim_{\tau\rightarrow0^{+}}Q_{n}(\tau)=0, \ \ \ \lim_{\tau\rightarrow\infty}Q_{n}(\tau)=1,
$$
there exists $\tau_{n}>0$ such that $Q_{n}(\tau_{n})=\frac{1}{2}$. Let us define $v_{n}:=u_{n}^{(\tau_{n})}=\tau^{\frac{N-2}{2}}_nu_{n}(\tau_n x)$. Hence
$$
\int_{\mathbb{R}^N}\int_{\mathbb{R}^N}
\frac{|v_{n}(y)|^{2_{\alpha, \mu}^{\ast}}|v_{n}(x)|^{2_{\alpha, \mu}^{\ast}}}{|x|^{\alpha}|x-y|^{\mu}|y|^{\alpha}}
dydx=1, \ \ \ \int_{\mathbb{R}^N}|\nabla v_{n}|^{2}dx\rightarrow S_{\alpha, \mu}
$$
and
\begin{equation}\label{c0}
\frac{1}{2}=\sup_{z\in\mathbb{R}^N}\int_{B(z,1)}\Big(\int_{\mathbb{R}^N}\frac{
|u_{n}(y)|^{2_{\alpha, \mu}^{\ast}}}{|x-y|^{\mu}|y|^{\alpha}}dy\Big)
\frac{|u_{n}(x)|^{2_{\alpha, \mu}^{\ast}}}{|x|^{\alpha}}dx.
\end{equation}
Since $\{v_{n}\}$ is a bounded sequence in $D^{1,2}(\mathbb{R}^N)$, we may assume, going if necessary to a subsequence,
\begin{center}
$v_{n}\rightharpoonup v$ in $D^{1,2}(\mathbb{R}^N)$,
 \end{center}
\begin{center}
$|\nabla v_{n}|^{2}\rightharpoonup \omega+|\nabla v|^{2}$,
\end{center}
\begin{center}
$\Big(\displaystyle\int_{\mathbb{R}^N}\frac{
|v_{n}(y)|^{2_{\alpha, \mu}^{\ast}}}{|x-y|^{\mu}|y|^{\alpha}}dy\Big)
\frac{|v_{n}(x)|^{2_{\alpha, \mu}^{\ast}}}{|x|^{\alpha}}\rightharpoonup
\nu+\Big(\displaystyle\int_{\mathbb{R}^N}\frac{
|v(y)|^{2_{\alpha, \mu}^{\ast}}}{|x-y|^{\mu}|y|^{\alpha}}dy\Big)
\frac{|v(x)|^{2_{\alpha, \mu}^{\ast}}}{|x|^{\alpha}}$,
\end{center}
\begin{center}
$v_{n}\rightarrow v$ a.e. on $\mathbb{R}^N$.
\end{center}
By Lemma \ref{CCP1},
\begin{equation}\label{c1}
S_{\alpha, \mu}=\lim_{n\rightarrow\infty}\int_{\mathbb{R}^N}|\nabla v_{n}|^{2}dx=\int_{\mathbb{R}^N}|\nabla v|^{2}dx+\int_{\mathbb{R}^N}d\omega+\omega_{\infty},
\end{equation}
\begin{equation}\label{c2}
1=\int_{\mathbb{R}^N}\int_{\mathbb{R}^N}
\frac{|v_{n}(y)|^{2_{\alpha, \mu}^{\ast}}|v_{n}(x)|^{2_{\alpha, \mu}^{\ast}}}{|x|^{\alpha}|x-y|^{\mu}|y|^{\alpha}}
dydx=\int_{\mathbb{R}^N}\int_{\mathbb{R}^N}
\frac{|v(y)|^{2_{\alpha, \mu}^{\ast}}|v(x)|^{2_{\alpha, \mu}^{\ast}}}{|x|^{\alpha}|x-y|^{\mu}|y|^{\alpha}}
dydx+\int_{\mathbb{R}^N}d\nu+\nu_{\infty},
\end{equation}
where
$$
\nu_{\infty}=\lim_{R\rightarrow\infty}\overline{\lim}_{n\rightarrow\infty}\int_{|x|\geq R}\Big(\int_{\mathbb{R}^N}\frac{
|v_{n}(y)|^{2_{\alpha, \mu}^{\ast}}}{|x-y|^{\mu}|y|^{\alpha}}dy\Big)
\frac{|v_{n}(x)|^{2_{\alpha, \mu}^{\ast}}}{|x|^{\alpha}}dx,
$$
and
$$
\omega_{\infty}=\lim_{R\rightarrow\infty}\overline{\lim}_{n\rightarrow\infty}\int_{|x|\geq R}|\nabla v_{n}|^{2}dx.
$$
We deduce from \eqref{c1}, \eqref{cp6}, \eqref{cp71} and the definition of $S_{\alpha, \mu}$ that
$$
S_{\alpha, \mu}\geq S_{\alpha, \mu}\Big(\int_{\mathbb{R}^N}\int_{\mathbb{R}^N}
\frac{|v(y)|^{2_{\alpha, \mu}^{\ast}}|v(x)|^{2_{\alpha, \mu}^{\ast}}}{|x|^{\alpha}|x-y|^{\mu}|y|^{\alpha}}
dydx\Big)^{\frac{1}{2_{\alpha, \mu}^{\ast}}}+S_{\alpha, \mu}\Big(\int_{\mathbb{R}^N}d\nu\Big)^{\frac{1}{2_{\alpha, \mu}^{\ast}}}
+S_{\alpha, \mu}\nu_{\infty}^{\frac{2}{2_{\alpha, \mu}^{\ast}}}.
$$
It follows from \eqref{c2} that $\displaystyle\int_{\mathbb{R}^N}\int_{\mathbb{R}^N}
\frac{|v(y)|^{2_{\alpha, \mu}^{\ast}}|v(x)|^{2_{\alpha, \mu}^{\ast}}}{|x|^{\alpha}|x-y|^{\mu}|y|^{\alpha}}
dydx$, $\displaystyle\int_{\mathbb{R}^N}d\nu$ and $\nu_{\infty}$ are equal either to 0 or to 1. By \eqref{c0}, $\nu_{\infty}\leq\frac{1}{2}$ so that $\nu_{\infty}=0$. If $\displaystyle\int_{\mathbb{R}^N}d\nu=1$ then $v=0$ and $\displaystyle\int_{\mathbb{R}^N}d\omega
\leq S_{\alpha, \mu}\left(\int_{\mathbb{R}^N}d\nu\right)^{\frac{1}{2_{\alpha, \mu}^{\ast}}}$. The Lemma \ref{CCP1} implies that $\nu$ is concentrated at a single point $z_{0}$. We deduce from \eqref{c0} the contradiction
$$
\aligned
\frac{1}{2}&=\sup_{z\in\mathbb{R}^N}\int_{B(z,1)}\Big(\int_{\mathbb{R}^N}\frac{
|u_{n}(y)|^{2_{\alpha, \mu}^{\ast}}}{|x-y|^{\mu}|y|^{\alpha}}dy\Big)
\frac{|u_{n}(x)|^{2_{\alpha, \mu}^{\ast}}}{|x|^{\alpha}}dx\\
&\geq\int_{B(z_{0},1)}\Big(\int_{\mathbb{R}^N}\frac{
|u_{n}(y)|^{2_{\alpha, \mu}^{\ast}}}{|x-y|^{\mu}|y|^{\alpha}}dy\Big)
\frac{|u_{n}(x)|^{2_{\alpha, \mu}^{\ast}}}{|x|^{\alpha}}dx\rightarrow\int_{\mathbb{R}^N}d\nu=1.
\endaligned
$$
Thus
$$
\int_{\mathbb{R}^N}\int_{\mathbb{R}^N}
\frac{|v_{n}(y)|^{2_{\alpha, \mu}^{\ast}}|v_{n}(x)|^{2_{\alpha, \mu}^{\ast}}}{|x|^{\alpha}|x-y|^{\mu}|y|^{\alpha}}
dydx=\int_{\mathbb{R}^N}\int_{\mathbb{R}^N}
\frac{|v(y)|^{2_{\alpha, \mu}^{\ast}}|v(x)|^{2_{\alpha, \mu}^{\ast}}}{|x|^{\alpha}|x-y|^{\mu}|y|^{\alpha}}
dydx=1
$$
and so
$$
S_{\alpha, \mu}=\lim_{n\rightarrow\infty}\int_{\mathbb{R}^N}|\nabla v_{n}|^{2}dx=\int_{\mathbb{R}^N}|\nabla v|^{2}dx.
$$
$\hfill{} \Box$

\subsection{Regularity}

By Theorem 4.5 of \cite{CLO1} we know that equation \eqref{Intcase} is equivalent to the following integral system
\begin{equation}\label{ASD}
\begin{cases}
\displaystyle u(x)=\int_{\mathbb{R}^{N}}\frac{v(y)u(y)^{{2}^{\ast}_{\alpha,\mu}-1}}{|x-y|^{N-2}}dy,\vspace{3mm}\\
\displaystyle v(x)=\int_{\mathbb{R}^{N}}\frac{u(y)^{2^{\ast}_{\alpha, \mu}}}{|x|^{\alpha}|x-y|^{\mu}|y|^{\alpha}}dy.
\end{cases}
\end{equation}
Similar Hardy-Littlewood-Sobolev integral system with $\alpha=0$ has been studied in other papers. Li and Ma \cite{LM} studied the integral system
\begin{equation}\label{SimilarSys2}
\begin{cases}
\displaystyle u(x)=\int_{\mathbb{R}^{N}}\frac{u^p(y)v^q(y)}{|x-y|^{N-\tau}}dy,\vspace{3mm}\\
\displaystyle v(x)=\int_{\mathbb{R}^{N}}\frac{u^q(y)v^p(y)}{|x-y|^{N-\tau}}dy.
\end{cases}
\end{equation}
The authors proved the radial symmetry of positive solutions and obtained the uniqueness results. Jin and Li \cite{JL} obtained the
optimal integral intervals for the solutions of system
\begin{equation}\label{SimilarSys}
\begin{cases}
\displaystyle u(x)=\int_{\mathbb{R}^{N}}\frac{v^q(y)}{|x-y|^{N-\tau}}dy,\vspace{3mm}\\
\displaystyle v(x)=\int_{\mathbb{R}^{N}}\frac{u^p(y)}{|x-y|^{N-\tau}}dy,
\end{cases}
\end{equation}
which is associated with the study of the sharp constant of the Hardy-Littlewood-Sobolev inequality, Chen, Li and Ou used the method of moving planes in integral forms to prove the radial symmetry of $u, v$. Later
Lei \cite{Le} studied
\begin{equation}\label{SimilarSys2}
\begin{cases}
\displaystyle u(x)=\int_{\mathbb{R}^{N}}\frac{v(y)u(y)}{|x-y|^{N-\tau}}dy,\vspace{3mm}\\
\displaystyle v(x)=\int_{\mathbb{R}^{N}}\frac{u(y)^{2}}{|x-y|^{\mu}}dy.
\end{cases}
\end{equation}
He obtained the integrability
result for the integrable solution $u$ of integral systems and proved that the
integrable solution decays quickly and the nonintegrable bounded solution decays almost slowly. For special $\tau$, the result implies the corresponding results of the positive solutions of the static Hartree equation. For the weighted integral system,
 Chen et al. \cite{CJLL} considered
\begin{equation}\label{ASD2}
\begin{cases}
\displaystyle u(x)=\int_{\mathbb{R}^{N}}\frac{v^q(y)}{|x|^{\beta}|x-y|^{\mu}|y|^{\alpha}}dy,\vspace{3mm}\\
\displaystyle v(x)=\int_{\mathbb{R}^{N}}\frac{u^p(y)}{|x|^{\alpha}|x-y|^{\mu}|y|^{\beta}}dy.
\end{cases}
\end{equation}
They obtained the symmetry, monotonicity and regularity of solutions of \eqref{ASD2}.
Chen, Liu and Lu \cite{CLL} considered
\begin{equation}\label{ASD3}
\begin{cases}
\displaystyle u(x)=\int_{\mathbb{R}^{N}}\frac{f_1(u,v)}{|x|^{\beta}|x-y|^{\mu}|y|^{\alpha}}dy,\vspace{3mm}\\
\displaystyle v(x)=\int_{\mathbb{R}^{N}}\frac{f_2(u,v)}{|x|^{\alpha}|x-y|^{\mu}|y|^{\beta}}dy,
\end{cases}
\end{equation}
where
$$
f_1(u,v)=\lambda_1u^{p_1}+\mu_1u^{q_1}+\nu_1u^{\alpha_1}v^{\beta_1},\\
f_2(u,v)=\lambda_2u^{p_2}+\mu_2u^{q_2}+\nu_2u^{\alpha_2}v^{\beta_2}.
$$
Under $p_i,q_i>1$, $i=1,2$ and additional assumptions $\alpha_1,\beta_2>0$, $\alpha_2,\beta_1>1$ they proved the integrability, even $C^{\infty}$ regularity of $u, v$.  Furthermore, if $p_i,q_i,\alpha_i,\beta_i>1$, $i=1,2$, then the solutions are symmetric and strictly decreasing about the origin.

In the following we will study the integral system \eqref{ASD}, that is $\alpha\neq0$. Although the integral system with weights has been studied in \cite{CLL, CJLL}, but the results there do not include \eqref{ASD} as a special case. In fact, different from the equations \eqref{ASD2} and  \eqref{ASD3} where the assumptions of the integrability of $u,v$ are independent, the initial integrability of $v$ relies on $u$ and depends on the parameters $N, \mu, \al$ greatly. The integrability and $C^{\infty}$ regularity results in Theorems \ref{QWEQ}, \ref{ZX} and \ref{DF} can be rewritten in the following equivalent form.
\begin{thm}\label{QWE2}
Assume that $N\geq3$, $\alpha\geq0$, $0<\mu<N$, $0<2\alpha+\mu\leq N$. Let $(u,v)\in L^{\frac{2N}{N-2}}(\mathbb{R}^{N})\times L^{\frac{2N}{2\alpha+\mu}}(\mathbb{R}^{N})$ be a pair of positive solutions of system \eqref{ASD}.

$(C1)$.~If~$N=3,4,5,6$ and $N-2\leq2\alpha+\mu\leq\min\{N,4\}$, then $(u,v)\in L^{p}(\mathbb{R}^{N})\times L^{q}(\mathbb{R}^{N})$ with
$$
p\in\left(\frac{N}{N-2},+\infty\right)~~and~~q\in\left(\frac{2N}{N-2+2\alpha+\mu},\frac{2N}{2+2\alpha+\mu-N}\right).
$$

$(C2)$.~If~$N=5,6$ and $4<2\alpha+\mu\leq N$ while $N\geq7$ and $N-2\leq2\alpha+\mu\leq N$, then $(u,v)\in L^{p}(\mathbb{R}^{N})\times L^{q}(\mathbb{R}^{N})$ with
$$
p\in\left(\frac{N}{N-2},\frac{2N}{2\alpha+\mu-4}\right)~~and~~q\in\left(\frac{2N}{N+2\alpha+\mu-2},\frac{2N}{2(2\alpha+\mu)-N-2}\right).
$$

$(C3)$.~If~$N=3,4,5,6$ and $0<2\alpha+\mu<N-2$ while $N\geq7$ and $0\leq2\alpha+\mu\leq 4$ or $\frac{N+2}{2}\leq2\alpha+\mu<N-2$, then $(u,v)\in L^{p}(\mathbb{R}^{N})\times L^{q}(\mathbb{R}^{N})$ with
$$
p\in\left(\frac{2N}{N-2+2\alpha+\mu},\frac{2N}{N-2-2\alpha-\mu}\right)~~and~~q\in\left(\frac{N}{2\alpha+\mu},+\infty\right).
$$

$(C4)$.~If~$N\geq7$ and $4<2\alpha+\mu<\frac{N+2}{2}$, then $(u,v)\in L^{p}(\mathbb{R}^{N})\times L^{q}(\mathbb{R}^{N})$ with
$$
p\in\left(\frac{2N}{N-2+2\alpha+\mu},\frac{2N}{2\alpha+\mu-4}\right)~~and~~q\in\left(\frac{N}{2\alpha+\mu},\frac{2N}{2(2\alpha+\mu)-N-2}\right).
$$
\end{thm}
\begin{thm}\label{ZX2}
Assume that $N=3,4,5,6$, $0\leq\alpha<2$, $0<\mu<N$ and $N-2\leq2\alpha+\mu\leq\min\{N,4\}$. Let $(u,v)\in L^{\frac{2N}{N-2}}(\mathbb{R}^{N})\times L^{\frac{2N}{2\alpha+\mu}}(\mathbb{R}^{N})$ be a pair of positive solutions of system \eqref{ASD}, then $|x|^{\alpha}v(x)\in L^{\infty}(\mathbb{R}^{N})$ and $u(x)\in L^{\infty}(\mathbb{R}^{N})$.
\end{thm}
\begin{thm}\label{DF2}
Assume that $N=3,4,5,6$, $0\leq\alpha<2$, $0<\mu<N$ and $N-2\leq2\alpha+\mu\leq\min\{N,4\}$. Let $(u,v)\in L^{\frac{2N}{N-2}}(\mathbb{R}^{N})\times L^{\frac{2N}{2\alpha+\mu}}(\mathbb{R}^{N})$ be a pair of positive solutions of system \eqref{ASD}, then $u(x)\in C^{\infty}(\mathbb{R}^{N}-\{0\})$.
\end{thm}

  We First introduce the regularity lifting lemma (see \cite{MCL}).
\begin{thm}\label{ABC1}
Let $V$ be a topological vector space and $X:=\{v\in V:\|v\|_{X}<\infty\}$, $Y:=\{v\in V:\|v\|_{Y}<\infty\}$. Assume $T$ be a contraction map from $X\rightarrow X$ and $Y\rightarrow Y$, $f\in X$ and there exists a function $g\in X\cap Y$ such that $f=Tf+g$ in $X$. Then $f\in X\cap Y$.
\end{thm}
For any constant $A>0$, we define
\begin{equation}\nonumber
u_{A}(x)=
\begin{cases}
u(x),~~|u(x)|>A~or~|x|>A,\\
0,~~~~~~|u(x)|\leq A~and~|x|\leq A,
\end{cases}
\end{equation}
and $u_{B}(x)=u(x)-u_{A}(x)$. Define the functions
$$
F_{u}(t)=\int_{\mathbb{R}^{N}}\frac{u(y)^{{2}^{\ast}_{\alpha,\mu}-1}t(y)}{|x-y|^{N-2}}dy,~~~
G_{u}(s)=\int_{\mathbb{R}^{N}}\frac{u(y)^{2^{\ast}_{\alpha, \mu}-1}s(y)}{|x|^{\alpha}|x-y|^{\mu}|y|^{\alpha}}dy
$$
and an operator $T_{A}:L^{p}\times L^{q}\rightarrow L^{p}\times L^{q}$,
$$
T_{A}(s,t)=(F_{u_{A}}(t),G_{u_{A}}(s))
$$
with $||(s,t)||_{L^{p}\times L^{q}}=||s||_{L^{p}}+||t||_{L^{q}}$, then there holds
$$
(u,v)=T_{A}(u,v)+(F_{u_{B}}(v),G_{u_{B}}(u)).
$$
\begin{lem}\label{ABCDE}
Assume that $p,q$ satisfy $p\in\left(\frac{N}{N-2},+\infty\right)\cap\left(\frac{2N}{N-2+2\alpha+\mu},+\infty\right)$, $2N>p(N-2-2\alpha-\mu)$ and $\frac{1}{p}-\frac{1}{q}=\frac{N-2-2\alpha-\mu}{2N}$. For $A$ sufficiently large, $T_{A}$ is a contraction map from $L^{p}(\mathbb{R}^{N})\times L^{q}(\mathbb{R}^{N})$ to itself.
\end{lem}
\begin{proof}
Actually, there exists constant $C_{1}>0$,
\begin{equation}\nonumber
\begin{aligned}
\|F_{u_{A}}(t)\|_{L^{p}}&=\left\|\int_{\mathbb{R}^{N}}\frac{u_{A}^{2^{\ast}_{\alpha, \mu}-1}t}{|x-y|^{N-2}}dy\right\|_{L^{p}}\\
&\leq C_{1}\left\|u_{A}^{2^{\ast}_{\alpha, \mu}-1}t\right\|_{L^{\frac{Np}{N+2p}}}\\
&\leq C_{1}\|u_{A}\|^{{2}^{\ast}_{\alpha,\mu}-1}_{L^{2^{\ast}}}\|t\|_{L^{q}},
\end{aligned}
\end{equation}
since $p>\frac{N}{N-2}$, $2N>p(N-2-2\alpha-\mu)$ and $\frac{1}{p}-\frac{1}{q}=\frac{N-2-2\alpha-\mu}{2N}$.

Similarly, there exists constant $C_{2}>0$,
\begin{equation}\nonumber
\begin{aligned}
\|G_{u_{A}}(s)\|_{L^{q}}&=\left\|\int_{\mathbb{R}^{N}}\frac{u_{A}^{{2}^{\ast}_{\alpha,\mu}-1}s}{|x|^{\alpha}|x-y|^{\mu}|y|^{\alpha}}dy\right\|_{L^{q}}\\
&\leq C_{2}\left\|u_{A}^{{2}^{\ast}_{\alpha,\mu}-1}s\right\|_{L^{\frac{2Np}{2N+p(N+2-2\alpha-\mu)}}}\\
&\leq C_{2}\|u_{A}\|^{{2}^{\ast}_{\alpha,\mu}-1}_{L^{{2}^{\ast}}}\|s\|_{L^{p}},
\end{aligned}
\end{equation}
since $p>\frac{2N}{N-2+2\alpha+\mu}$.

Since $u\in L^{{2}^{\ast}}(\mathbb{R}^{N})$, we can choose $A$ is sufficiently large and constant $C=\max\{C_{1},C_{2}\}$ such that $C||u_{A}\|^{{2}^{\ast}_{\alpha}-1}_{L^{{2}^{\ast}}}\leq\frac{1}{2}$,
\begin{equation}\nonumber
\begin{aligned}
||T_{A}(s,t)||_{L^{p}\times L^{q}}
&=||F_{u_{A}}(t)||_{L^{p}}+||G_{u_{A}}(s)||_{L^{q}}\\
&\leq\frac{1}{2}(\|t\|_{L^{q}}+\|s\|_{L^{p}})\\
&=\frac{1}{2}\|(s,t)\|_{L^{p}\times L^{q}}.
\end{aligned}
\end{equation}
Hence $T_{A}$ is a contraction map from $L^{p}(\mathbb{R}^{N})\times L^{q}(\mathbb{R}^{N})$ to itself.
\end{proof}
{\bf Proof of Theorem \ref{QWE2}.} By Lemma \ref{ABCDE}, we know that, for $A$ is sufficiently large, $T_{A}$ is a contraction map from $L^{p}(\mathbb{R}^{N})\times L^{q}(\mathbb{R}^{N})$ to itself.
Next we show $(F_{u_{B}}(v),G_{u_{B}}(u))\in L^{p}(\mathbb{R}^{N})\times L^{q}(\mathbb{R}^{N})$. It is obvious that $|u_{B}|\leq A$ and $u_{B}=0$ when $|x|>A$, consequently,
\begin{equation}\nonumber
\begin{aligned}
\|F_{u_{B}}(v)\|_{L^{p}}&=\left\|\int_{\mathbb{R}^{N}}\frac{u_{B}^{{2}^{\ast}_{\alpha,\mu}-1}v}{|x-y|^{N-2}}dy\right\|_{L^{p}}\\
&\leq C_{1}\left\|u_{B}^{{2}^{\ast}_{\alpha,\mu}-1}v\right\|_{L^{\frac{Np}{N+2p}}}\\
&\leq C_{1}\left\|u_{B}^{{2}^{\ast}_{\alpha,\mu}-1}\right\|_{L^{\frac{2Np}{2N+p(4-2\alpha-\mu)}}}\|v\|_{L^{\frac{2N}{2\alpha+\mu}}}
\end{aligned}
\end{equation}
and
\begin{equation}\nonumber
\begin{aligned}
\|G_{u_{B}}(u)\|_{L^{q}}
&=\left\|\int_{\mathbb{R}^{N}}\frac{u_{b}^{{2}^{\ast}_{\alpha,\mu}-1}u}{|x|^{\alpha}|x-y|^{\mu}|y|^{\alpha}}dy\right\|_{L^{q}}\\
&\leq C_{2}\left\|u_{B}^{{2}^{\ast}_{\alpha,\mu}-1}u\right\|_{L^{\frac{2Np}{2N+p(N+2-2\alpha-\mu)}}}\\
&\leq C_{2}\left\|u_{B}^{{2}^{\ast}_{\alpha,\mu}-1}\right\|_{L^{\frac{2Np}{2N+p(4-2\alpha-\mu)}}}\|u\|_{L^{\frac{2N}{N-2}}},
\end{aligned}
\end{equation}
since $2N>p(2\alpha+\mu-4)$.

In conclusion,  we can conclude from Theorem \ref{ABC1} that, if $p$ and $q$ satisfy
\begin{equation}\nonumber
\begin{cases}
\displaystyle p>\frac{N}{N-2},\vspace{3mm}\\
\displaystyle p>\frac{2N}{N-2+2\alpha+\mu},\vspace{3mm}\\
\displaystyle 2N>p(N-2-2\alpha-\mu),\vspace{3mm}\\
\displaystyle 2N>p(2\alpha+\mu-4),\vspace{3mm}\\
\displaystyle \frac{1}{p}-\frac{1}{q}=\frac{N-2-2\alpha-\mu}{2N},
\end{cases}
\end{equation}
then $(u,v)\in(L^{\frac{2N}{N-2}}(\mathbb{R}^{N})\times L^{\frac{2N}{2\alpha+\mu}}(\mathbb{R}^{N}))\cap(L^{p}(\mathbb{R}^{N})\times L^{q}(\mathbb{R}^{N}))$.

If $\frac{1}{p}-\frac{1}{q}=\frac{N-2-2\alpha-\mu}{2N}$, then we have the following classification:

$(1)$.~If~$N-2-2\alpha-\mu\leq0$ and $2\alpha+\mu-4\leq0$, then $(u,v)\in L^{p}(\mathbb{R}^{N})\times L^{q}(\mathbb{R}^{N})$ with
$$
p>\frac{N}{N-2}~~\hbox{and}~~\frac{2N}{N-2+2\alpha+\mu}<q<\frac{2N}{2+2\alpha+\mu-N}.$$

$(2)$.~If~$N-2-2\alpha-\mu\leq0$ and $2\alpha+\mu-4>0$, then $(u,v)\in L^{p}(\mathbb{R}^{N})\times L^{q}(\mathbb{R}^{N})$ with
$$
\frac{N}{N-2}<p<\frac{2N}{2\alpha+\mu-4}~~and~~\frac{2N}{N+2\alpha+\mu-2}<q<\frac{2N}{2(2\alpha+\mu)-N-2}.
$$

$(3)$.~If~$N-2-2\alpha-\mu>0$ and $2\alpha+\mu-4\leq0$, then $(u,v)\in L^{p}(\mathbb{R}^{N})\times L^{q}(\mathbb{R}^{N})$ with
$$
\frac{2N}{N-2+2\alpha+\mu}<p<\frac{2N}{N-2-2\alpha-\mu}~~and~~q>\frac{N}{2\alpha+\mu}.
$$

$(4)$.~If~$N-2-2\alpha-\mu>0$, $2\alpha+\mu-4>0$ and $N+2-2(2\alpha+\mu)\leq0$, then $(u,v)\in L^{p}(\mathbb{R}^{N})\times L^{q}(\mathbb{R}^{N})$ with
$$
\frac{2N}{N-2+2\alpha+\mu}<p<\frac{2N}{N-2-2\alpha-\mu}~~and~~q>\frac{N}{2\alpha+\mu}.
$$

$(5)$.~If~$N-2-2\alpha-\mu>0$, $2\alpha+\mu-4>0$ and $N+2-2(2\alpha+\mu)>0$, then $(u,v)\in L^{p}(\mathbb{R}^{N})\times L^{q}(\mathbb{R}^{N})$ with
$$
\frac{2N}{N-2+2\alpha+\mu}<p<\frac{2N}{2\alpha+\mu-4}~~and~~\frac{N}{2\alpha+\mu}<q<\frac{2N}{2(2\alpha+\mu)-N-2}.
$$

More accurately, since that $0<2\alpha+\mu\leq N$, if $\frac{1}{p}-\frac{1}{q}=\frac{N-2-2\alpha-\mu}{2N}$, we have different regularity lifting results for any pair of solution $(u,v)$ of system \eqref{ASD}:

$(C1)$.~If~$N=3,4,5,6$ and $N-2\leq2\alpha+\mu\leq\min\{N,4\}$, then $(u,v)\in L^{p}(\mathbb{R}^{N})\times L^{q}(\mathbb{R}^{N})$ with
$$
p\in\left(\frac{N}{N-2},+\infty\right)~~and~~q\in\left(\frac{2N}{N-2+2\alpha+\mu},\frac{2N}{2+2\alpha+\mu-N}\right).
$$

$(C2)$.~If~$N=5,6$ and $4<2\alpha+\mu\leq N$ while $N\geq7$ and $N-2\leq2\alpha+\mu\leq N$, then $(u,v)\in L^{p}(\mathbb{R}^{N})\times L^{q}(\mathbb{R}^{N})$ with
$$
p\in\left(\frac{N}{N-2},\frac{2N}{2\alpha+\mu-4}\right)~~and~~q\in\left(\frac{2N}{N+2\alpha+\mu-2},\frac{2N}{2(2\alpha+\mu)-N-2}\right).
$$

$(C3)$.~If~$N=3,4,5,6$ and $0<2\alpha+\mu<N-2$ while $N\geq7$ and $0\leq2\alpha+\mu\leq 4$ or $\frac{N+2}{2}\leq2\alpha+\mu<N-2$, then $(u,v)\in L^{p}(\mathbb{R}^{N})\times L^{q}(\mathbb{R}^{N})$ with
$$
p\in\left(\frac{2N}{N-2+2\alpha+\mu},\frac{2N}{N-2-2\alpha-\mu}\right)~~and~~q\in\left(\frac{N}{2\alpha+\mu},+\infty\right).
$$

$(C4)$.~If~$N\geq7$ and $4<2\alpha+\mu<\frac{N+2}{2}$, then $(u,v)\in L^{p}(\mathbb{R}^{N})\times L^{q}(\mathbb{R}^{N})$ with
$$
p\in\left(\frac{2N}{N-2+2\alpha+\mu},\frac{2N}{2\alpha+\mu-4}\right)~~and~~q\in\left(\frac{N}{2\alpha+\mu},\frac{2N}{2(2\alpha+\mu)-N-2}\right).
$$
$\hfill{} \Box$

{\bf Proof of Theorem \ref{ZX2}.}
Firstly, we prove that $|x|^{\alpha}v(x)\in L^{\infty}(\mathbb{R}^{N})$. By \eqref{ASD},
$$
|x|^{\alpha}v(x)=\int_{\mathbb{R}^{N}}\frac{u(y)^{2^{\ast}_{\alpha, \mu}}}{|x-y|^{\mu}|y|^{\alpha}}dy,
$$
then for any $r>0$, we have
\begin{equation}\label{UJM}
||x|^{\alpha}v(x)|\leq\int_{B_{r}(0)}\frac{|u(y)|^{2^{\ast}_{\alpha, \mu}}}{|x-y|^{\mu}|y|^{\alpha}}dy+\int_{\mathbb{R}^{N}-B_{r}(0)}\frac{|u(y)|^{2^{\ast}_{\alpha, \mu}}}{|x-y|^{\mu}|y|^{\alpha}}dy.
\end{equation}

On the one hand, for $x\in\mathbb{R}^{N}-B_{2r}(0)$, we have $|x-y|>|y|$,
\begin{equation}\nonumber
\begin{aligned}
\int_{B_{r}(0)}\frac{|u(y)|^{2^{\ast}_{\alpha, \mu}}}{|x-y|^{\mu}|y|^{\alpha}}dy
&<\int_{B_{r}(0)}\frac{|u(y)|^{2^{\ast}_{\alpha, \mu}}}{|y|^{\mu+\alpha}}dy\\
&\leq\|u\|^{2^{\ast}_{\alpha, \mu}}_{L^{\frac{2^{\ast}_{\alpha, \mu}k}{k-1}}(B_{r}(0))}
\left\|\frac{1}{|y|^{\mu+\alpha}}\right\|_{L^{k}(B_{r}(0))}<\infty,
\end{aligned}
\end{equation}
where $1<k<\frac{N}{\mu+\alpha}$.
For $x\in B_{2r}(0)$,
\begin{equation}\nonumber
\begin{aligned}
\int_{B_{r}(0)}\frac{|u(y)|^{2^{\ast}_{\alpha, \mu}}}{|x-y|^{\mu}|y|^{\alpha}}dy
&\leq\int_{B_{r}(0)}\frac{|u(y)|^{2^{\ast}_{\alpha, \mu}}}{|y|^{\mu+\alpha}}dy
+\int_{B_{3r}(x)}\frac{|u(y)|^{2^{\ast}_{\alpha, \mu}}}{|x-y|^{\mu+\alpha}}dy<\infty.
\end{aligned}
\end{equation}
Hence we can deduce that
\begin{equation}\label{EDC}
\int_{B_{r}(0)}\frac{|u(y)|^{2^{\ast}_{\alpha, \mu}}}{|x-y|^{\mu}|y|^{\alpha}}dy<\infty.
\end{equation}

On the other hand,
$$
\int_{\mathbb{R}^{N}-B_{r}(0)}\frac{|u(y)|^{2^{\ast}_{\alpha, \mu}}}{|x-y|^{\mu}|y|^{\alpha}}dy
=\int_{(\mathbb{R}^{N}-B_{r}(0))\cap B_{r}(x)}\frac{|u(y)|^{2^{\ast}_{\alpha, \mu}}}{|x-y|^{\mu}|y|^{\alpha}}dy
+\int_{(\mathbb{R}^{N}-B_{r}(0))\cap (\mathbb{R}^{N}-B_{r}(x))}\frac{|u(y)|^{2^{\ast}_{\alpha, \mu}}}{|x-y|^{\mu}|y|^{\alpha}}dy.
$$
As the preceding estimates, we know
\begin{equation}\nonumber
\begin{aligned}
\int_{(\mathbb{R}^{N}-B_{r}(0))\cap B_{r}(x)}\frac{|u(y)|^{2^{\ast}_{\alpha, \mu}}}{|x-y|^{\mu}|y|^{\alpha}}dy
&\leq\frac{1}{r^{\alpha}}\int_{(\mathbb{R}^{N}-B_{r}(0))\cap B_{r}(x)}\frac{|u(y)|^{2^{\ast}_{\alpha, \mu}}}{|x-y|^{\mu}}dy\\
&\leq\frac{1}{r^{\alpha}}\int_{B_{r}(x)}\frac{|u(y)|^{2^{\ast}_{\alpha, \mu}}}{|x-y|^{\mu}}dy<\infty.\\
\end{aligned}
\end{equation}
We also estimate the other part on unbounded domain as
\begin{equation}\nonumber
\begin{aligned}
\int_{(\mathbb{R}^{N}-B_{r}(0))\cap (\mathbb{R}^{N}-B_{r}(x))}\frac{|u(y)|^{2^{\ast}_{\alpha, \mu}}}{|x-y|^{\mu}|y|^{\alpha}}dy
&\leq\frac{1}{r^{\mu}}\int_{\mathbb{R}^{N}-B_{r}(0)}\frac{|u(y)|^{2^{\ast}_{\alpha, \mu}}}{|y|^{\alpha}}dy\\
&\leq\frac{1}{r^{\mu}}\|u\|^{2^{\ast}_{\alpha, \mu}}_{L^{\frac{2^{\ast}_{\alpha, \mu}k}{k-1}}(\mathbb{R}^{N}-B_{r}(0))}
\left\|\frac{1}{|y|^{\alpha}}\right\|_{L^{k}(\mathbb{R}^{N}-B_{r}(0))}<\infty,
\end{aligned}
\end{equation}
where $k\geq\frac{N}{\alpha}$.
Therefore we can conclude that
\begin{equation}\label{YHN}
\int_{\mathbb{R}^{N}-B_{r}(0)}\frac{|u(y)|^{2^{\ast}_{\alpha, \mu}}}{|x-y|^{\mu}|y|^{\alpha}}dy<\infty.
\end{equation}
Through \eqref{UJM}, \eqref{EDC}, \eqref{YHN}, we can obtain that
\begin{equation}\label{SDF}
|x|^{\alpha}v(x)\in L^{\infty}(\mathbb{R}^{N}).
\end{equation}

Secondly, we claim that $u(x)\in L^{\infty}(\mathbb{R}^{N})$. From \eqref{ASD} and \eqref{SDF}, we have
\begin{equation}\label{AS}
|u(x)|\leq\int_{\mathbb{R}^{N}}\frac{|y|^{\alpha}|v(y)||u(y)|^{{2}^{\ast}_{\alpha,\mu}-1}}{|x-y|^{N-2}|y|^{\alpha}}dy
\leq |||x|^{\alpha}v(x)||_{L^{\infty}(\mathbb{R}^{N})}\int_{\mathbb{R}^{N}}\frac{|u(y)|^{{2}^{\ast}_{\alpha,\mu}-1}}{|x-y|^{N-2}|y|^{\alpha}}dy.
\end{equation}
For any $r>0$, we decompose that
\begin{equation}\label{NM}
\int_{\mathbb{R}^{N}}\frac{|u(y)|^{{2}^{\ast}_{\alpha,\mu}-1}}{|x-y|^{N-2}|y|^{\alpha}}dy
=\int_{B_{r}(0)}\frac{|u(y)|^{2^{\ast}_{\alpha, \mu}-1}}{|x-y|^{N-2}|y|^{\alpha}}dy
+\int_{\mathbb{R}^{N}-B_{r}(0)}\frac{|u(y)|^{2^{\ast}_{\alpha, \mu}-1}}{|x-y|^{N-2}|y|^{\alpha}}dy.
\end{equation}
To estimate the first term, for $x\in\mathbb{R}^{N}-B_{2r}(0)$,
\begin{equation}\nonumber
\begin{aligned}
\int_{B_{r}(0)}\frac{|u(y)|^{2^{\ast}_{\alpha, \mu}-1}}{|x-y|^{N-2}|y|^{\alpha}}dy
&<\int_{B_{r}(0)}\frac{|u(y)|^{2^{\ast}_{\alpha, \mu}-1}}{|y|^{N-2+\alpha}}dy\\
&\leq\|u\|^{2^{\ast}_{\alpha, \mu}-1}_{L^{\frac{(2^{\ast}_{\alpha, \mu}-1)k}{k-1}}(B_{r}(0))}
\left\|\frac{1}{|y|^{N-2+\alpha}}\right\|_{L^{k}(B_{r}(0))}<\infty,
\end{aligned}
\end{equation}
where $1<k<\frac{N}{N-2+\alpha}$.
While for $x\in B_{2r}(0)$,
\begin{equation}\nonumber
\begin{aligned}
\int_{B_{r}(0)}\frac{|u(y)|^{2^{\ast}_{\alpha, \mu}-1}}{|x-y|^{N-2}|y|^{\alpha}}dy
&\leq\int_{B_{r}(0)}\frac{|u(y)|^{2^{\ast}_{\alpha, \mu}-1}}{|y|^{N-2+\alpha}}dy
+\int_{B_{3r}(x)}\frac{|u(y)|^{2^{\ast}_{\alpha, \mu}-1}}{|x-y|^{N-2+\alpha}}dy<\infty.
\end{aligned}
\end{equation}
Consequently we have
\begin{equation}\label{XC}
\int_{B_{r}(0)}\frac{|u(y)|^{2^{\ast}_{\alpha, \mu}-1}}{|x-y|^{N-2}|y|^{\alpha}}dy<\infty.
\end{equation}
To estimate the second term,,
\begin{equation}\label{VB}
\begin{aligned}
&\int_{\mathbb{R}^{N}-B_{r}(0)}\frac{|u(y)|^{2^{\ast}_{\alpha, \mu}-1}}{|x-y|^{N-2}|y|^{\alpha}}dy\\
&=\int_{(\mathbb{R}^{N}-B_{r}(0))\cap B_{r}(x)}\frac{|u(y)|^{2^{\ast}_{\alpha, \mu}-1}}{|x-y|^{N-2}|y|^{\alpha}}dy
+\int_{(\mathbb{R}^{N}-B_{r}(0))\cap (\mathbb{R}^{N}-B_{r}(x))}\frac{|u(y)|^{2^{\ast}_{\alpha, \mu}-1}}{|x-y|^{N-2}|y|^{\alpha}}dy.
\end{aligned}
\end{equation}
It is obvious that
\begin{equation}\label{VB}
\begin{aligned}
\int_{(\mathbb{R}^{N}-B_{r}(0))\cap B_{r}(x)}\frac{|u(y)|^{2^{\ast}_{\alpha, \mu}-1}}{|x-y|^{N-2}|y|^{\alpha}}dy
&\leq\frac{1}{r^{\alpha}}\int_{B_{r}(x)}\frac{|u(y)|^{2^{\ast}_{\alpha, \mu}-1}}{|x-y|^{N-2}}dy<\infty.
\end{aligned}
\end{equation}
Meanwhile, for the other integral in \eqref{VB},
\begin{equation}\nonumber
\begin{aligned}
\int_{(\mathbb{R}^{N}-B_{r}(0))\cap (\mathbb{R}^{N}-B_{r}(x))}\frac{|u(y)|^{2^{\ast}_{\alpha, \mu}-1}}{|x-y|^{N-2}|y|^{\alpha}}dy
&\leq\int_{\mathbb{R}^{N}-B_{r}(0)}\frac{|u(y)|^{2^{\ast}_{\alpha, \mu}-1}}{|y|^{N-2+\alpha}}dy
+\int_{\mathbb{R}^{N}-B_{r}(x)}\frac{|u(y)|^{2^{\ast}_{\alpha, \mu}-1}}{|x-y|^{N-2+\alpha}}dy.
\end{aligned}
\end{equation}
Since
\begin{equation}\nonumber
\begin{aligned}
\int_{\mathbb{R}^{N}-B_{r}(0)}\frac{|u(y)|^{2^{\ast}_{\alpha, \mu}-1}}{|y|^{N-2+\alpha}}dy
&\leq\|u\|^{2^{\ast}_{\alpha, \mu}-1}_{L^{\frac{(2^{\ast}_{\alpha, \mu}-1)k}{k-1}}(\mathbb{R}^{N}-B_{r}(0))}
\left\|\frac{1}{|y|^{N-2+\alpha}}\right\|_{L^{k}(\mathbb{R}^{N}-B_{r}(0))}<\infty,
\end{aligned}
\end{equation}
where  $k\geq\frac{N}{N-2+\alpha}$ if $0<2\alpha+\mu\leq2$
and  $\frac{N}{N-2+\alpha}<k<\frac{N}{2\alpha+\mu-2}$ if $2\alpha+\mu>2$,
we conclude that
\begin{equation}\label{CV}
\int_{(\mathbb{R}^{N}-B_{r}(0))\cap (\mathbb{R}^{N}-B_{r}(x))}\frac{|u(y)|^{2^{\ast}_{\alpha, \mu}-1}}{|x-y|^{N-2}|y|^{\alpha}}dy<\infty.
\end{equation}
Combining with \eqref{VB} and \eqref{CV}, we deduce that
\begin{equation}\label{BN}
\int_{\mathbb{R}^{N}-B_{r}(0)}\frac{|u(y)|^{2^{\ast}_{\alpha, \mu}-1}}{|x-y|^{N-2}|y|^{\alpha}}dy<\infty.
\end{equation}
Through \eqref{AS}, \eqref{NM}, \eqref{XC}, \eqref{BN}, we know that
$$
u(x)\in L^{\infty}(\mathbb{R}^{N}).
$$
$\hfill{} \Box$

{\bf Proof of Theorem \ref{DF2}.}
For any $x\in\mathbb{R}^{N}-\{0\}$, there exists $r<\frac{|x|}{2}$ such that
$$
u(x)=\int_{B_{2r}(x)}\frac{v(y)u(y)^{{2}^{\ast}_{\alpha,\mu}-1}}{|x-y|^{N-2}}dy+\int_{\mathbb{R}^{N}-B_{2r}(x)}\frac{v(y)u(y)^{{2}^{\ast}_{\alpha,\mu}-1}}{|x-y|^{N-2}}dy.
$$
One can see \cite[Chapter 10]{LL}, that for any $\delta<2$,
\begin{equation}\label{CDE}
\int_{B_{2r}(x)}\frac{v(y)u(y)^{{2}^{\ast}_{\alpha,\mu}-1}}{|x-y|^{N-2}}dy\in C^{\delta}(\mathbb{R}^{N}-\{0\}).
\end{equation}
Next we need to show that
$$
\int_{\mathbb{R}^{N}-B_{2r}(x)}\frac{v(y)u(y)^{{2}^{\ast}_{\alpha,\mu}-1}}{|x-y|^{N-2}}dy\in C^{\infty}(\mathbb{R}^{N}-\{0\}).
$$
Denote by
$$
\psi(x)=\int_{\mathbb{R}^{N}}\frac{v(y)u(y)^{{2}^{\ast}_{\alpha,\mu}-1}}{|x-y|^{N-2}}\chi_{\{\mathbb{R}^{N}-B_{2r}(x)\}}dy,
$$
we claim that $$\psi(x)\in C^{1}(\mathbb{R}^{N}-\{0\}).$$

In fact, for any small $t<r$, $0<\theta<1$ and $e_{i}$ is the unit $i$th vector,
\begin{equation}\nonumber
\begin{aligned}
\left|\frac{\psi(x+te_{i})-\psi(x)}{t}\right|
&\leq\frac{1}{|t|}\int_{\mathbb{R}^{N}}\left|\frac{v(y)u(y)^{{2}^{\ast}_{\alpha,\mu}-1}}{|x+te_{i}-y|^{N-2}}\chi_{\{\mathbb{R}^{N}-B_{2r}(x+te_{i})\}}
-\frac{v(y)u(y)^{{2}^{\ast}_{\alpha,\mu}-1}}{|x-y|^{N-2}}\chi_{\{\mathbb{R}^{N}-B_{2r}(x)\}}\right|dy\\
&\leq C\int_{\mathbb{R}^{N}}\frac{|v(y)||u(y)|^{{2}^{\ast}_{\alpha,\mu}-1}}{|x+\theta te_{i}-y|^{N-1}}\chi_{\{\mathbb{R}^{N}-B_{2r}(x+\theta te_{i})\}}dy\\
&\leq C\int_{B_{r}(0)}\frac{|v(y)||u(y)|^{{2}^{\ast}_{\alpha,\mu}-1}}{|x-y|^{N-1}}dy
+C\int_{\mathbb{R}^{N}-B_{r}(x)- B_{r}(0)}\frac{|v(y)||u(y)|^{{2}^{\ast}_{\alpha,\mu}-1}}{|x-y|^{N-1}}dy.\\
\end{aligned}
\end{equation}
The boundedness of $|x|^{\alpha}v(x)$ and $u(x)$ implies that
\begin{equation}\label{ER}
\begin{aligned}
\int_{B_{r}(0)}\frac{|v(y)||u(y)|^{{2}^{\ast}_{\alpha,\mu}-1}}{|x-y|^{N-1}}dy
&\leq \frac{1}{r^{N-1}}|||x|^{\alpha}v||_{L^{\infty}(\mathbb{R}^{N})}||u||^{{2}^{\ast}_{\alpha,\mu}-1}_{L^{\infty}(\mathbb{R}^{N})}\int_{B_{r}(0)}\frac{1}{|y|^{\alpha}}dy<\infty.
\end{aligned}
\end{equation}
We also have
\begin{equation}\nonumber
\begin{aligned}
&\int_{\mathbb{R}^{N}-B_{r}(x)- B_{r}(0)}\frac{|v(y)||u(y)|^{{2}^{\ast}_{\alpha,\mu}-1}}{|x-y|^{N-1}}dy\\
&\leq|||x|^{\alpha}v||_{L^{\infty}(\mathbb{R}^{N})}
\int_{\mathbb{R}^{N}-B_{r}(0)}\frac{|u(y)|^{{2}^{\ast}_{\alpha,\mu}-1}}{|y|^{N-1+\alpha}}dy
+|||x|^{\alpha}v||_{L^{\infty}(\mathbb{R}^{N})}
\int_{\mathbb{R}^{N}-B_{r}(x)}\frac{|u(y)|^{{2}^{\ast}_{\alpha,\mu}-1}}{|x-y|^{N-1+\alpha}}dy.
\end{aligned}
\end{equation}
For the following four cases:
(1). If $0<2\alpha+\mu\leq2$, $\alpha<1$, then $k>\frac{N}{N-1+\alpha}$;
(2).  If $0<2\alpha+\mu\leq2$, $1\leq\alpha<2$, then $k>1$;
(3).  If $2\alpha+\mu>2$, $\alpha<1$, then $\frac{N}{N-1+\alpha}<k<\frac{N}{2\alpha+\mu-2}$;
(4).  If $2\alpha+\mu>2$, $1\leq\alpha<2$, then $1<k<\frac{N}{2\alpha+\mu-2}$,
we have
\begin{equation}\nonumber
\begin{aligned}
\int_{\mathbb{R}^{N}-B_{r}(0)}\frac{|u(y)|^{2^{\ast}_{\alpha, \mu}-1}}{|y|^{N-1+\alpha}}dy
&\leq\|u(y)\|^{2^{\ast}_{\alpha, \mu}-1}_{L^{\frac{(2^{\ast}_{\alpha, \mu}-1)k}{k-1}}(\mathbb{R}^{N}-B_{r}(0))}
\left\|\frac{1}{|y|^{N-1+\alpha}}\right\|_{L^{k}(\mathbb{R}^{N}-B_{r}(0))}<\infty.
\end{aligned}
\end{equation}
Therefore we can conclude that
\begin{equation}\label{WE}
\int_{\mathbb{R}^{N}-B_{r}(x)- B_{r}(0)}\frac{|v(y)||u(y)|^{{2}^{\ast}_{\alpha,\mu}-1}}{|x-y|^{N-1}}dy<\infty.
\end{equation}
Hence by the \eqref{ER}, \eqref{WE} and Lebesgue's Dominated Convergence Theorem, we claim that $\psi(x)\in C^{1}(\mathbb{R}^{N}-\{0\})$.
Continuing this process, we can improve $\psi(x)\in C^{\infty}(\mathbb{R}^{N}-\{0\})$. Combining with \eqref{CDE}, we conclude that $u(x)\in C^{\delta}(\mathbb{R}^{N}-\{0\})$. By the classical bootstrap technique \cite[Chapter 10]{LL}, we prove that $u(x)\in C^{\infty}(\mathbb{R}^{N}-\{0\})$.$\hfill{} \Box$

\subsection{Symmetry}
In this subsection, we will use the moving plane method in integral forms to prove Theorem \ref{thm1.3}. The moving plane method was invented by
Alexandrov in the 1950s and has been further developed by many people. Among the existing results, Chen
et al. \cite{CLO1} applied the moving plane method to integral equations and obtained the symmetry,
monotonicity and nonexistence properties of the solutions.
 By virtue of the HLS inequality or
the weighted HLS inequality, moving plane method in integral form can be used to explore various specific features
of the integral equation itself without the help of the maximum principle of
differential equation. We may refer the readers to \cite{CLL1, CLZ} for recent progress of the moving plane methods. Furthermore, the qualitative analysis of the solutions for nonlocal equations has been studied in \cite{CLO3, CL3, LY, Le, LLM}.

To investigate the symmetry of solutions for equation \eqref{Intcase}, we continue to study the equivalent integral system \eqref{ASD} and introduce the following symbols:
$$
x^{\lambda}=(2\lambda-x_{1},...,x_{N});~~u(x^{\lambda})=u_{\lambda}(x);~~v(x^{\lambda})=v_{\lambda}(x).
$$
Based on the symbols, we specify some planes:
$$
\Sigma_{\lambda}=\{x\in\mathbb{R}^{N}:x_{1}<\lambda\};~~T_{\lambda}=\{x\in\mathbb{R}^{N}:x_{1}=\lambda\};
$$
$$
\Sigma^{u}_{\lambda}=\{x\in\Sigma_{\lambda}|u(x)>u(x^{\lambda})\};~~\Sigma^{v}_{\lambda}=\{x\in\Sigma_{\lambda}|v(x)>v(x^{\lambda})\}.
$$
Next, we introduce two equality which are the basics of moving planes.
\begin{lem}\label{lem31}
For any pair of solution $(u,v)$ of system \eqref{ASD}, we have
\begin{equation}\label{C1}
u(x)-u_{\lambda}(x)=\int_{\Sigma_{\lambda}}\left[\frac{1}{|x-y|^{N-2}}-\frac{1}{|x^{\lambda}-y|^{N-2}}\right]\left[v(y)u(y)^{2_{\alpha, \mu}^{\ast}-1}-v(y^{\lambda})u(y^{\lambda})^{2_{\alpha, \mu}^{\ast}-1}\right]dy
\end{equation}
and
\begin{equation}\label{D1}
\begin{aligned}
&v(x)-v_{\lambda}(x)\\
&=\int_{\Sigma_{\lambda}}\frac{1}{|x-y|^{\mu}}\left(\frac{u(y)^{2_{\alpha, \mu}^{\ast}}}{|x|^{\alpha}|y|^{\alpha}}-\frac{u(y^{\lambda})^{2_{\alpha, \mu}^{\ast}}}{|x^{\lambda}|^{\alpha}|y^{\lambda}|^{\alpha}}\right)
+\frac{1}{|x^{\lambda}-y|^{\mu}}\left(\frac{u(y^{\lambda})^{2_{\alpha, \mu}^{\ast}}}{|x|^{\alpha}|y^{\lambda}|^{\alpha}}-\frac{u(y)^{2_{\alpha, \mu}^{\ast}}}{|x^{\lambda}|^{\alpha}|y|^{\alpha}}\right)dy.
\end{aligned}
\end{equation}
\end{lem}

\begin{proof}
Direct computing shows
\begin{equation}\nonumber
\begin{aligned}
u(x)&=\int_{\mathbb{R}^{N}}\frac{v(y)u(y)^{2_{\alpha, \mu}^{\ast}-1}}{|x-y|^{N-2}}dy\\
&=\int_{\Sigma_{\lambda}}\frac{v(y)u(y)^{2_{\alpha, \mu}^{\ast}-1}}{|x-y|^{N-2}}dy
+\int_{\mathbb{R}^{N}-\Sigma_{\lambda}}\frac{v(y)u(y)^{2_{\alpha, \mu}^{\ast}-1}}{|x-y|^{N-2}}dy\\
&=\int_{\Sigma_{\lambda}}\frac{v(y)u(y)^{2_{\alpha, \mu}^{\ast}-1}}{|x-y|^{N-2}}dy
+\int_{\Sigma_{\lambda}}\frac{v(y^{\lambda})u(y^{\lambda})^{2_{\alpha, \mu}^{\ast}-1}}{|x-y^{\lambda}|^{N-2}}dy
\end{aligned}
\end{equation}
and
$$
u_{\lambda}(x)=\int_{\Sigma_{\lambda}}\frac{v(y)u(y)^{2_{\alpha, \mu}^{\ast}-1}}{|x-y^{\lambda}|^{N-2}}dy
+\int_{\Sigma_{\lambda}}\frac{v(y^{\lambda})u(y^{\lambda})^{2_{\alpha, \mu}^{\ast}-1}}{|x-y^{\lambda}|^{N-2}}dy.
$$
Since $|x^{\lambda}-y^{\lambda}|=|x-y|$ and $|x-y^{\lambda}|=|x^{\lambda}-y|$, hence it's obvious that
$$
u(x)-u_{\lambda}(x)=\int_{\Sigma_{\lambda}}\left[\frac{1}{|x-y|^{N-2}}-\frac{1}{|x^{\lambda}-y|^{N-2}}\right]\left[v(y)u(y)^{2_{\alpha, \mu}^{\ast}-1}-v(y^{\lambda})u(y^{\lambda})^{2_{\alpha, \mu}^{\ast}-1}\right]dy.
$$
Then we have a similiar result for $v(x)$:
$$
v(x)=\int_{\Sigma_{\lambda}}\frac{u(y)^{2_{\alpha, \mu}^{\ast}}}{|x|^{\alpha}|x^{\lambda}-y|^{\mu}|y|^{\alpha}}dy
+\int_{\Sigma_{\lambda}}\frac{u(y^{\lambda})^{2_{\alpha, \mu}^{\ast}}}{|x|^{\alpha}|x^{\lambda}-y^{\lambda}|^{\mu}|y^{\lambda}|^{\alpha}}dy
$$
and
$$
v_{\lambda}(x)=\int_{\Sigma_{\lambda}}\frac{u(y)^{2_{\alpha, \mu}^{\ast}}}{|x^{\lambda}|^{\alpha}|x^{\lambda}-y|^{\mu}|y|^{\alpha}}dy
+\int_{\Sigma_{\lambda}}\frac{u(y^{\lambda})^{2_{\alpha, \mu}^{\ast}}}{|x^{\lambda}|^{\alpha}|x^{\lambda}-y^{\lambda}|^{\mu}|y^{\lambda}|^{\alpha}}dy.
$$
There holds
$$
v(x)-v_{\lambda}(x)=\int_{\Sigma_{\lambda}}\frac{1}{|x-y|^{\mu}}\left(\frac{u(y)^{2_{\alpha, \mu}^{\ast}}}{|x|^{\alpha}|y|^{\alpha}}-\frac{u(y^{\lambda})^{2_{\alpha, \mu}^{\ast}}}{|x^{\lambda}|^{\alpha}|y^{\lambda}|^{\alpha}}\right)
+\frac{1}{|x^{\lambda}-y|^{\mu}}\left(\frac{u(y^{\lambda})^{2_{\alpha, \mu}^{\ast}}}{|x|^{\alpha}|y^{\lambda}|^{\alpha}}-\frac{u(y)^{2_{\alpha, \mu}^{\ast}}}{|x^{\lambda}|^{\alpha}|y|^{\alpha}}\right)dy.
$$
\end{proof}
Consider the case $2\alpha+\mu\neq4$, we can get the following estimate:
\begin{lem}\label{lem51}
Suppose that $\alpha\geq0$, $0<\mu<N$, $2\alpha+\mu\leq3$ if $N=3$ while $2\alpha+\mu<4$ if $N\geq4$. For any $\lambda<0$, there exists a constant $C>0$ such that:
$$
\|u-u_{\lambda}\|_{L^{2^{\ast}}(\Sigma^{u}_{\lambda})}\leq C\left[||u||^{2_{\alpha, \mu}^{\ast}}_{L^{2^{\ast}}(\mathbb{R}^{N})}
||u||^{2_{\alpha, \mu}^{\ast}-2}_{L^{2^{\ast}}(\Sigma^{u}_{\lambda})}
+||u||^{2_{\alpha, \mu}^{\ast}-1}_{L^{2^{\ast}}(\Sigma^{v}_{\lambda})}||u||^{2_{\alpha, \mu}^{\ast}-1}_{L^{2^{\ast}}(\Sigma^{u}_{\lambda})}\right]||u-u_{\lambda}||_{L^{2^{\ast}}(\Sigma^{u}_{\lambda})}.
$$
\end{lem}
\begin{proof}
In order to estimate $u(x)-u_{\lambda}(x)$ by equation \eqref{C1}, we should divide the domain of integration into four disjoint parts as
$$
\Sigma_{\lambda}=\{\Sigma^{u}_{\lambda}\cap\Sigma^{v}_{\lambda}\}\cup\{\Sigma^{u}_{\lambda}-\Sigma^{v}_{\lambda}\}\cup\{\Sigma^{v}_{\lambda}-\Sigma^{u}_{\lambda}\}\cup
\{\Sigma_{\lambda}-\Sigma^{u}_{\lambda}-\Sigma^{v}_{\lambda}\}.
$$
In $\Sigma^{u}_{\lambda}\cap\Sigma^{v}_{\lambda}$, we use the Mean Value Theorem to get
\begin{equation}\nonumber
\begin{aligned}
vu^{2_{\alpha, \mu}^{\ast}-1}-v_{\lambda}u_{\lambda}^{2_{\alpha, \mu}^{\ast}-1}
&\leq(2_{\alpha, \mu}^{\ast}-1)vu^{2_{\alpha, \mu}^{\ast}-2}(u-u_{\lambda})+u^{2_{\alpha, \mu}^{\ast}-1}(v-v_{\lambda}).
\end{aligned}
\end{equation}
In $\Sigma^{u}_{\lambda}-\Sigma^{v}_{\lambda}$, we use the Mean Value Theorem again to get
\begin{equation}\nonumber
\begin{aligned}
vu^{2_{\alpha, \mu}^{\ast}-1}-v_{\lambda}u_{\lambda}^{2_{\alpha, \mu}^{\ast}-1}
&\leq(2_{\alpha, \mu}^{\ast}-1)vu^{2_{\alpha, \mu}^{\ast}-2}(u-u_{\lambda}).
\end{aligned}
\end{equation}
In $\Sigma^{v}_{\lambda}-\Sigma^{u}_{\lambda}$,
\begin{equation}\nonumber
\begin{aligned}
vu^{2_{\alpha, \mu}^{\ast}-1}-v_{\lambda}u_{\lambda}^{2_{\alpha, \mu}^{\ast}-1}
&\leq u^{2_{\alpha, \mu}^{\ast}-1}(v-v_{\lambda}).
\end{aligned}
\end{equation}
In $\Sigma_{\lambda}-\Sigma^{u}_{\lambda}-\Sigma^{v}_{\lambda}$,
\begin{equation}\nonumber
\begin{aligned}
vu^{2_{\alpha, \mu}^{\ast}-1}-v_{\lambda}u_{\lambda}^{2_{\alpha, \mu}^{\ast}-1}&\leq0.
\end{aligned}
\end{equation}
Hence lemma \ref{lem31} implies when $x\in\Sigma_{\lambda}$,
\begin{equation}\nonumber
\begin{aligned}
u(x)-u_{\lambda}(x)
&\leq(2_{\alpha, \mu}^{\ast}-1)\int_{\Sigma^{u}_{\lambda}}\frac{vu^{2_{\alpha, \mu}^{\ast}-2}(u-u_{\lambda})}{|x-y|^{N-2}}
dy+\int_{\Sigma^{v}_{\lambda}}\frac{u^{2_{\alpha, \mu}^{\ast}-1}(v-v_{\lambda})}{|x-y|^{N-2}}dy.
\end{aligned}
\end{equation}
From equation \eqref{D1} we get that when $x\in\Sigma_{\lambda}$,
\begin{equation}\nonumber
\begin{aligned}
v(x)-v_{\lambda}(x)&=\int_{\Sigma_{\lambda}}\frac{1}{|x-y|^{\mu}}\left(\frac{u(y)^{2_{\alpha, \mu}^{\ast}}}{|x|^{\alpha}|y|^{\alpha}}-\frac{u(y^{\lambda})^{2_{\alpha, \mu}^{\ast}}}{|x^{\lambda}|^{\alpha}|y^{\lambda}|^{\alpha}}\right)
+\frac{1}{|x^{\lambda}-y|^{\mu}}\left(\frac{u(y^{\lambda})^{2_{\alpha, \mu}^{\ast}}}{|x|^{\alpha}|y^{\lambda}|^{\alpha}}-\frac{u(y)^{2_{\alpha, \mu}^{\ast}}}{|x^{\lambda}|^{\alpha}|y|^{\alpha}}\right)dy\\
&\leq\int_{\Sigma_{\lambda}}\frac{1}{|x|^{\alpha}}\left(\frac{1}{|x-y|^{\mu}}-\frac{1}{|x^{\lambda}-y|^{\mu}}\right)\left(\frac{u(y)^{2_{\alpha, \mu}^{\ast}}}{|y|^{\alpha}}-\frac{u(y^{\lambda})^{2_{\alpha, \mu}^{\ast}}}{|y^{\lambda}|^{\alpha}}\right)dy\\
&\leq\int_{\Sigma^{u}_{\lambda}}\frac{u(y)^{2_{\alpha, \mu}^{\ast}}-u(y^{\lambda})^{2_{\alpha, \mu}^{\ast}}}{|x|^{\alpha}|x-y|^{\mu}|y|^{\alpha}}dy.
\end{aligned}
\end{equation}
We can use the Mean Value Theorem to show:
\begin{equation}\nonumber
v(x)-v_{\lambda}(x)\leq 2_{\alpha, \mu}^{\ast}\int_{\Sigma^{u}_{\lambda}}\frac{u(y)^{2_{\alpha, \mu}^{\ast}-1}[u(y)-u(y^{\lambda})]}{|x|^{\alpha}|x-y|^{\mu}|y|^{\alpha}}dy.
\end{equation}
Then the H\"{o}lder inequality and the weighted HLS inequality implies
\begin{equation}\label{G1}
\begin{aligned}
||v-v_{\lambda}||_{L^{\frac{2N}{2\alpha+\mu}}(\Sigma^{v}_{\lambda})}
&\leq C||u||^{2_{\alpha, \mu}^{\ast}-1}_{L^{2^{\ast}}(\Sigma^{u}_{\lambda})}||u-u_{\lambda}||_{L^{2^{\ast}}(\Sigma^{u}_{\lambda})}.
\end{aligned}
\end{equation}
In virtue of the H\"{o}lder inequality and the HLS inequality, it follows that
\begin{equation}\nonumber
\begin{aligned}
\left\|\int_{\Sigma^{u}_{\lambda}}\frac{vu^{2_{\alpha, \mu}^{\ast}-2}(u-u_{\lambda})}{|x-y|^{N-2}}
dy\right\|_{L^{2^{\ast}}(\Sigma^{u}_{\lambda})}
&\leq C||vu^{2_{\alpha, \mu}^{\ast}-2}(u-u_{\lambda})||_{L^{\frac{2N}{N+2}}(\Sigma^{u}_{\lambda})}\\
&\leq C||u||^{2_{\alpha, \mu}^{\ast}}_{L^{2^{\ast}}(\mathbb{R}^{N})}
||u||^{2_{\alpha, \mu}^{\ast}-2}_{L^{2^{\ast}}(\Sigma^{u}_{\lambda})}
||u-u_{\lambda}||_{L^{2^{\ast}}(\Sigma^{u}_{\lambda})}.
\end{aligned}
\end{equation}
On the other hand, the H\"{o}lder inequality and \eqref{G1} implies
\begin{equation}\nonumber
\begin{aligned}
\left\|\int_{\Sigma^{v}_{\lambda}}\frac{u^{2_{\alpha, \mu}^{\ast}-1}(v-v_{\lambda})}{|x-y|^{N-2}}dy\right\|_{L^{2^{\ast}}(\Sigma^{u}_{\lambda})}
&\leq C||u^{2_{\alpha, \mu}^{\ast}-1}(v-v_{\lambda})||_{L^{\frac{2N}{N+2}}(\Sigma^{v}_{\lambda})}\\
&\leq C||u||^{2_{\alpha, \mu}^{\ast}-1}_{L^{2^{\ast}}(\Sigma^{v}_{\lambda})}||u||^{2_{\alpha, \mu}^{\ast}-1}_{L^{2^{\ast}}(\Sigma^{u}_{\lambda})}||u-u_{\lambda}||_{L^{2^{\ast}}(\Sigma^{u}_{\lambda})}.
\end{aligned}
\end{equation}
Hence there exists a constant $C>0$,
\begin{equation}\nonumber
\begin{aligned}
||u-u_{\lambda}||_{L^{2^{\ast}}({\Sigma^{u}_{\lambda}})}
&\leq C\left[||u||^{2_{\alpha, \mu}^{\ast}}_{L^{2^{\ast}}(\mathbb{R}^{N})}
||u||^{2_{\alpha, \mu}^{\ast}-2}_{L^{2^{\ast}}(\Sigma^{u}_{\lambda})}
+||u||^{2_{\alpha, \mu}^{\ast}-1}_{L^{2^{\ast}}(\Sigma^{v}_{\lambda})}||u||^{2_{\alpha, \mu}^{\ast}-1}_{L^{2^{\ast}}(\Sigma^{u}_{\lambda})}\right]||u-u_{\lambda}||_{L^{2^{\ast}}(\Sigma^{u}_{\lambda})}.
\end{aligned}
\end{equation}
\end{proof}

For the case $2\alpha+\mu=4$, by similar estimates one may find that lemma \ref{lem51} should be replace by
\begin{lem}\label{lemABC}
Suppose that $N\geq4$, $\alpha\geq0$, $0<\mu<N$ and $2\alpha+\mu=4$, then for any $\lambda<0$, there exists a constant $C>0$ such that:
$$
\|u-u_{\lambda}\|_{L^{2^{\ast}}(\Sigma^{u}_{\lambda})}\leq C\left[||v||^{2}_{L^{\frac{N}{2}}(\Sigma^{u}_{\lambda})}
+||u||_{L^{2^{\ast}}(\Sigma^{v}_{\lambda})}||u||_{L^{2^{\ast}}(\Sigma^{u}_{\lambda})}\right]||u-u_{\lambda}||_{L^{2^{\ast}}(\Sigma^{u}_{\lambda})}.
$$
\end{lem}

The integral inequality in Lemma \ref{lem51} and \ref{lemABC} will allows us to carry out the moving plane method in integral forms to prove the main results.

We will first show that, for $\lambda$ is sufficiently negative,
\begin{equation}\label{E1}
\begin{aligned}
u(x)\leq u(x^{\lambda}),~~~v(x)\leq v(x^{\lambda}),~~\forall x\in\Sigma_{\lambda}.
\end{aligned}
\end{equation}
Then we can start moving the plane from near $-\infty$ to the right as long as \eqref{E1} holds.
\begin{lem}\label{lem61}
Under the assumption of Theorem \ref{thm1.3}, there exists $\lambda_{0}<0$ such that for any $\lambda\leq\lambda_{0}$, $u(x)\leq u(x^{\lambda})$ and $v(x)\leq v(x^{\lambda})$ hold in $\Sigma_{\lambda}$.
\end{lem}
\begin{proof}
Since $u(x)$ and $v(x)$ is integrable, we can choose $\lambda_{0}$ is sufficiently negative, such that for $\lambda\leq\lambda_{0}$, we have
\begin{equation}\label{J1}
C\left[||u||^{2_{\alpha, \mu}^{\ast}}_{L^{2^{\ast}}(\mathbb{R}^{N})}
||u||^{2_{\alpha, \mu}^{\ast}-2}_{L^{2^{\ast}}(\Sigma^{u}_{\lambda})}
+||u||^{2_{\alpha, \mu}^{\ast}-1}_{L^{2^{\ast}}(\Sigma^{v}_{\lambda})}||u||^{2_{\alpha, \mu}^{\ast}-1}_{L^{2^{\ast}}(\Sigma^{u}_{\lambda})}\right]<1.
\end{equation}
Then Lemma \ref{lem51} implies that
$$
||u-u_{\lambda}||_{L^{2^{\ast}}(\Sigma^{u}_{\lambda})}=0.
$$
Therefore we conclude that $\Sigma^{u}_{\lambda}$ must be empty, that is $u(x)\leq u(x^{\lambda})$ in $\Sigma_{\lambda}$, from inequality \eqref{G1}, we find $\Sigma^{v}_{\lambda}$ is also empty, hence $v(x)\leq v(x^{\lambda})$.
\end{proof}
We now move the plane $T_{\lambda}$ to the right as long as \eqref{E1} holds. If we define
$$
\lambda_{1}=\sup\{~\lambda~|~u(x)\leq u(x^{\mu})~,~v(x)\leq v(x^{\mu})~,~x\in\Sigma_{\mu}~,~\mu\leq\lambda\},
$$
then we must have $\lambda_{1}<\infty$ because of applying a similar argument for $\lambda$ near $+\infty$. Since the plane can be moved from $-\infty$, we show that
\begin{lem}\label{lem71}
Under the assumption of Theorem \ref{thm1.3}, then for any $\lambda_{1}<0$, there holds $u(x)\equiv u(x^{\lambda_{1}})$ and $v(x)\equiv v(x^{\lambda_{1}})$ in $\Sigma_{\lambda_{1}}$.
\end{lem}
\begin{proof}
Suppose $u(x)\not\equiv u(x^{\lambda_{1}})$, from lemma \ref{lem61} and equation \eqref{C1}, we can achieve $u(x)<u_{\lambda}(x)$ in $\Sigma_{\lambda_{1}}$. For any sufficient small $\eta>0$, we can choose $R$ is sufficiently large, such that
$$
C\left[||u||^{2_{\alpha, \mu}^{\ast}}_{L^{2^{\ast}}(\mathbb{R}^{N}-B_{R}(0))}
||u||^{2_{\alpha, \mu}^{\ast}-2}_{L^{2^{\ast}}(\mathbb{R}^{N}-B_{R}(0))}
+||u||^{2_{\alpha, \mu}^{\ast}-1}_{L^{2^{\ast}}(\mathbb{R}^{N}-B_{R}(0))}||u||^{2_{\alpha, \mu}^{\ast}-1}_{L^{2^{\ast}}(\mathbb{R}^{N}-B_{R}(0))}\right]<\eta.
$$
For any fixed $\varepsilon>0$, we set
$$
P_{\varepsilon}=\{x\in\Sigma_{\lambda_{1}}\cap B_{R}(0)|u(x^{\lambda_{1}})-u(x)>\varepsilon\},~~Q_{\varepsilon}=\{x\in\Sigma_{\lambda_{1}}\cap B_{R}(0)|u(x^{\lambda_{1}})-u(x)\leq\varepsilon\}.
$$
For $\lambda$ is sufficiently close to $\lambda_{1}$, we fix a narrow domain $\Omega_{\lambda}=(\Sigma_{\lambda}-\Sigma_{\lambda_{1}})\cap B_{R}(0)$, then for any $x\in\Sigma^{u}_{\lambda}\cap P_{\varepsilon}$, it is obvious that $u(x^{\lambda_{1}})-u(x^{\lambda})>u(x^{\lambda_{1}})-u(x)>\varepsilon$.
However when $\lambda\rightarrow\lambda_{1}$, the Chebyshev inequality implies
$$
\mathcal{L}(\Sigma^{u}_{\lambda}\cap P_{\varepsilon})\leq\mathcal{L}(\{x\in B_{R}(0)|u(x^{\lambda_{1}})-u(x^{\lambda})>\varepsilon\})\rightarrow0.
$$
Where $\mathcal{L}$ is the Lebesgue measure. Let $\varepsilon\rightarrow0$ and $\lambda\rightarrow\lambda_{1}$, we have
$$
\mathcal{L}(\Sigma^{u}_{\lambda}\cap B_{R}(0))\leq\mathcal{L}(\Sigma^{u}_{\lambda}\cap P_{\varepsilon})+\mathcal{L}(Q_{\varepsilon})+\mathcal{L}(\Omega_{\lambda})\rightarrow0.
$$
Therefore there exists $\tau>0$ such that for any $\lambda\in[\lambda_{1},\lambda_{1}+\tau]$, inequality \eqref{J1} holds.
As the preceding proof, we assert that for all $\lambda\in[\lambda_{1},\lambda_{1}+\tau]$, we have \eqref{E1}. This contradicts the definition of $\lambda_{1}$, then \eqref{C1} implies $v(x)\equiv v(x^{\lambda_{1}})$.
\end{proof}
{\flushleft\bf{Proof of Theorem \ref{thm1.3}}.}
Similarly, the plane can be moved from $+\infty$ to left, therefore we denote the corresponding parameter $\lambda_{1}^{\prime}$ satisfying
$$
\lambda^{\prime}_{1}=\inf\{~\lambda~|~u(x)\leq u(x^{\mu})~,~v(x)\leq v(x^{\mu})~,~x\in\Sigma^{\prime}_{\mu}~,~\mu\geq\lambda\},
$$
where $\Sigma^{\prime}_{\lambda}=\{x\in\mathbb{R}^{N}:x_{1}>\lambda\}$. If $\lambda^{\prime}_{1}>0$, $u(x)$ and $v(x)$ are also radially symmetric about $T_{\lambda^{\prime}_{1}}$ in the similar argument. Hence when $\lambda_{1}=\lambda^{\prime}_{1}\neq0$, we can deduce from Lemma \ref{lem71} that $u(x)\equiv u(x^{\lambda_{1}})$ and $v(x)\equiv v(x^{\lambda_{1}})$ in $\Sigma_{\lambda_{1}}$, that contradicts to \eqref{C1} because
\begin{equation}\nonumber
\begin{aligned}
0&=\int_{\Sigma_{\lambda_{1}}}\frac{1}{|x-y|^{\mu}}\left(\frac{u(y)^{2_{\alpha, \mu}^{\ast}}}{|x|^{\alpha}|y|^{\alpha}}-\frac{u(y^{\lambda_{1}})^{2_{\alpha, \mu}^{\ast}}}{|x^{\lambda_{1}}|^{\alpha}|y^{\lambda_{1}}|^{\alpha}}\right)
+\frac{1}{|x^{\lambda_{1}}-y|^{\mu}}\left(\frac{u(y^{\lambda_{1}})^{2_{\alpha, \mu}^{\ast}}}{|x|^{\alpha}|y^{\lambda_{1}}|^{\alpha}}-\frac{u(y)^{2_{\alpha, \mu}^{\ast}}}{|x^{\lambda_{1}}|^{\alpha}|y|^{\alpha}}\right)dy\\
&\leq\int_{\Sigma_{\lambda_{1}}}\frac{u(y)^{2_{\alpha, \mu}^{\ast}}}{|x|^{\alpha}}\left(\frac{1}{|x-y|^{\mu}}-\frac{1}{|x^{\lambda_{1}}-y|^{\mu}}\right)\left(\frac{1}{|y|^{\alpha}}-\frac{1}{|y^{\lambda_{1}}|^{\alpha}}\right)dy<0.
\end{aligned}
\end{equation}
However, this is impossible. Therefore, we have $\lambda_{1}=\lambda^{\prime}_{1}=0$, $u(x)$ and $v(x)$ are radially symmetric about $T_{0}$. In conclusion, $u(x)$ and $v(x)$ are radially symmetric about origin on $x_{1}$ direction. Since the directions are chosen arbitrarily, we deduce $u(x)$ and $v(x)$ are radially symmetric about origin.
$\hfill{} \Box$

Notice that the extremal function $u(x)$ satisfies
$$
\int_{\mathbb{R}^N}|\nabla u|^{2}dx=S_{\alpha, \mu}
$$
and
\begin{equation}\label{5.1}
\int_{\mathbb{R}^{N}}\int_{\mathbb{R}^{N}}\frac{|u(x)|^{2^{\ast}_{\alpha, \mu}}|u(y)|^{2^{\ast}_{\alpha, \mu}}}
{|x|^{\alpha}|x-y|^{\mu}|y|^{\alpha}}dxdy=1.
\end{equation}
In the following we are going to investigate the growth rate of $u(x)$ at infinity.
{\flushleft\bf{Proof of Theorem \ref{thm1.4}}.}
For any $R>0$, we can get
\begin{equation}\nonumber
\begin{aligned}
\int_{B_{R}(0)}\int_{B_{R}(0)}\frac{|u(x)|^{2^{\ast}_{\alpha, \mu}}|u(y)|^{2^{\ast}_{\alpha, \mu}}}
{|x|^{\alpha}|x-y|^{\mu}|y|^{\alpha}}dxdy
&\geq\frac{1}{2^{\mu }R^{\mu}}\int_{B_{R}(0)}\int_{B_{R}(0)}\frac{|u(x)|^{2^{\ast}_{\alpha, \mu}}|u(y)|^{2^{\ast}_{\alpha, \mu}}}{|x|^{\alpha}|y|^{\alpha}}dxdy\\
&\geq\frac{1}{2^{\mu}R^{\mu}}\left(\int_{B_{R}(0)}\frac{|u(x)|^{2^{\ast}_{\alpha, \mu}}}{|x|^{\alpha}}dx\right)^{2}.
\end{aligned}
\end{equation}
If $u$ is radially symmetric and decreasing then
\begin{equation}\nonumber
\begin{aligned}
\int_{\mathbb{R}^{N}}\int_{\mathbb{R}^{N}}\frac{|u(x)|^{2^{\ast}_{\alpha, \mu}}|u(y)|^{2^{\ast}_{\alpha, \mu}}}
{|x|^{\alpha}|x-y|^{\mu}|y|^{\alpha}}dxdy&\geq\frac{1}{2^{\mu}R^{\mu}}\left(\int_{B_{R}(0)}\frac{|u(x)|^{2^{\ast}_{\alpha, \mu}}}{|x|^{\alpha}}dx\right)^{2}\\
&\geq\frac{\omega^{2}_{N-1}|u(R)|^{2\cdot2^{\ast}_{\alpha, \mu}}}{(N-\alpha)^{2}2^{\mu}R^{\mu+2\alpha-2N}}.
\end{aligned}
\end{equation}
According to \eqref{5.1}, it is obvious that
$$
u(R)\leq\left[\frac{(N-\alpha)^{2}2^{\mu}R^{\mu+2\alpha-2N}}{\omega^{2}_{N-1}}\right]^{\frac{1}{2\cdot2^{\ast}_{\alpha, \mu}}}=\left[\frac{(N-\alpha)^{2}2^{\mu }}{\omega^{2}_{N-1}}\right]^{\frac{1}{2\cdot2^{\ast}_{\alpha, \mu}}}\left(\frac{1}{R}\right)^{\frac{N-2}{2}}.
$$
Therefore, for any $x\neq0$,
$$
u(|x|)\leq\left[\frac{(N-\alpha)^{2}2^{\mu}}{\omega^{2}_{N-1}}\right]^{\frac{1}{2\cdot2^{\ast}_{\alpha, \mu}}}\left(\frac{1}{|x|}\right)^{\frac{N-2}{2}}.
$$
$\hfill{} \Box$
\section{The subcritical case}

\subsection{Existence of ground states}

 Let $H^{1}(\mathbb{R}^{N})$ be the Sobolev space endowed with the standard norm $\|\cdot\|$.
 We are going to study the subcritical weighted Choquard equation of the form
\begin{equation}\label{SWE}
-\Delta u+u
=\frac{1}{|x|^{\alpha}}\left(\int_{\mathbb{R}^{N}}\frac{|u(y)|^{p}}
{|x-y|^{\mu}|y|^{\alpha}}dy\right)|u|^{p-2}u\hspace{4.14mm}\mbox{in}\hspace{1.14mm} \mathbb{R}^{N},
\end{equation}
where $N\geq3$, $0<\mu<N$, $\alpha\geq0$ and $0<2\alpha+\mu\leq N$, $2_{\alpha, \mu}^{\ast}=(2N-2\alpha-\mu)/(N-2)$ and  $2-\frac{2\alpha+\mu}{N}< p<2_{\alpha, \mu}^{\ast}$.
 Associated to equation \eqref{SWE} we have the $C^{1}$ energy functional $I:H^{1}(\mathbb{R}^{N})\rightarrow\mathbb{R}$ defined by
 \[
 I(u):=\frac{1}{2}\|u\|^{2}-\frac{1}{2p}\int_{\mathbb{R}^N}
\int_{\mathbb{R}^N}\frac{|u(x)|^{p}|u(y)|^{p}}
{|x|^{\alpha}|x-y|^{\mu}|y|^{\alpha}}dxdy.
 \]
In view of the weighted Hardy-Littlewood-Sobolev inequality, the energy functional is well defined. Furthermore, critical points of $I$ are solutions of equation \eqref{SWE}. In this subsection we are concerned with existence of \textit{positive ground state solutions}, that is, a positive solution which has minimal energy among all nontrivial solutions.

 We introduce the Nehari manifold associated to equation \eqref{SWE} by
\[
\mathcal{N}=\Big\{u\in H^{1}(\mathbb{R}^{N})\backslash\{0\}: <I^{\prime}(u), u>=0\Big\}.
\]
{\flushleft\bf{Proof of Theorem \ref{Existencesub}}.}
We can use Ekeland's variational principle  to obtain a sequence $(u_{n})\subset\mathcal{N}$ such that
\begin{equation}\label{paper1j6}
I(u_{n})\rightarrow c_{\mathcal{N}} \quad \mbox{and} \quad  I^{\prime}(u_{n})\rightarrow0.
\end{equation}
Let $H^{1}_{rad}(\mathbb{R}^{N})$ be the space of radial functions in $H^{1}(\mathbb{R}^{N})$. Note that we may suppose the minimizing sequence is positive and it is contained in the radial space. In fact, it follows that there exists a sequence $(t_{n})\subset(0,+\infty)$ such that $t_{n}|u_{n}|\subset\mathcal{N}$. Thus, one has
\[
c_{\mathcal{N}}\leq I(t_{n}|u_{n}|)\leq I(t_{n}u_{n})\leq\max_{t\geq0}I(tu_{n})=I(u_{n})=c_{\mathcal{N}}+o_{n}(1).
\]
Therefore, we may consider a nonnegative minimizing sequence. For the sake of simplicity we also denote $(u_{n})\subset\mathcal{N}$ with $u_{n}\geq0$. Let $(u_{n}^{*})$ be the Schwarz symmetrization sequence associated with $u_{n}$. By Schwarz symmetrization properties we deduce that
\[
\|u_{n}^{*}\|\leq \|u_{n}\|.
\]
Moreover, the nonlocal terms satisfies
\[
\int_{\mathbb{R}^N}
\int_{\mathbb{R}^N}\frac{|u_n(x)|^{p}|u_n(y)|^{p}}
{|x|^{\alpha}|x-y|^{\mu}|y|^{\alpha}}dxdy\leq \int_{\mathbb{R}^N}
\int_{\mathbb{R}^N}\frac{|u^{*}_n(x)|^{p}|u^{*}_n(y)|^{p}}
{|x|^{\alpha}|x-y|^{\mu}|y|^{\alpha}}dxdy.
\]
For more details regarding Schwarz symmetrization we refer the readers to \cite[Chapter 3]{LE1}. Let $(t_{n}^{*})\subset(0,+\infty)$ be such that $t_{n}^{*}u_{n}^{*}\subset\mathcal{N}$. It follows from the above estimates that
\[
c_{\mathcal{N}}\leq I(t_{n}^{*}u_{n}^{*})\leq I(t_{n}^{*}u_{n})\leq \max_{t\geq0}I(tu_{n})=I(u_{n})=c_{\mathcal{N}}+o_{n}(1).
\]
Therefore, we may consider a positive minimizing sequence $t_{n}^{*}u_{n}^{*}=u_{n}\subset\mathcal{N}$ satisfying \eqref{paper1j6}. Moreover, the minimizing sequence is bounded. Thus, up to a subsequence $u_{n}\rightharpoonup u_{0}$ weakly in $H^{1}_{rad}(\mathbb{R}^{N})$. By standard density arguments we may conclude that $I'(u_{0})=0$. If $u_{0}=0$, then
\[
	 \int_{\mathbb{R}^N}
\int_{\mathbb{R}^N}\frac{|u_n(x)|^{p}|u_n(y)|^{p}}
{|x|^{\alpha}|x-y|^{\mu}|y|^{\alpha}}dxdy\rightarrow \int_{\mathbb{R}^N}
\int_{\mathbb{R}^N}\frac{|u_{0}(x)|^{p}|u_{0}(y)|^{p}}
{|x|^{\alpha}|x-y|^{\mu}|y|^{\alpha}}dxdy=0, \quad \mbox{as} \hspace{0,2cm} n\rightarrow\infty.
	 \]
Thus we get that
$\|u_{n}\|\to 0
$
as $n$ to $\infty$, which contradicts with the fact that $u_{n}\subset\mathcal{N}$. Hence  $u_{0}\neq0$
and $u_{0}\in\mathcal{N}$ and $c_{\mathcal{N}}\leq I(u_{0})$. On the other hand, in light of Fatou's Lemma we obtain
\begin{eqnarray*}
	c_{\mathcal{N}}+o_{n}(1) & = & I(u_{n})-\frac{1}{2}<I^{\prime}(u_{n}), u_{n}>\\
	& = & \left(\frac{1}{2}-\frac{1}{2p} \right)\int_{\mathbb{R}^N}
\int_{\mathbb{R}^N}\frac{|u_n(x)|^{p}|u_n(y)|^{p}}
{|x|^{\alpha}|x-y|^{\mu}|y|^{\alpha}}dxdy\\
	& = & \left(\frac{1}{2}-\frac{1}{2p} \right)\int_{\mathbb{R}^N}
\int_{\mathbb{R}^N}\frac{|u_{0}(x)|^{p}|u_{0}(y)|^{p}}
{|x|^{\alpha}|x-y|^{\mu}|y|^{\alpha}}dxdy+o_{n}(1)\\
	& = & I(u_{0})+o_{n}(1),
\end{eqnarray*}
which implies that $u_{0}$ is a positive radial ground state solution for equation~\eqref{SWE}.$\hfill{} \Box$

\subsection{Poho\v{z}aev identity}
In this subsection we establish a Poho\v{z}aev identity for the weighted case. The Poho\v{z}aev identity for the subcritical Choquard equation without weight was obtained in \cite{MS1}.
\begin{lem}\label{QWE}
Assume that $N\geq3$, $\alpha\geq0$, $0<\mu<N$, $0<\mu+2\alpha\leq N$. For $u\in W^{2,2}_{loc}(\mathbb{R}^{N})\cap L^{\frac{2Np}{2N-2\al-\mu}}_{loc}(\mathbb{R}^{N})$ is a positive solution of \eqref{SWE}, then
$$
\frac{N-2}{2}\int_{\mathbb{R}^{N}}|\nabla u|^{2}dx+\frac{N}{2}\int_{\mathbb{R}^{N}}|u|^{2}dx
=\frac{2N-2\alpha-\mu}{2p}\int_{\mathbb{R}^{N}}\left(\int_{\mathbb{R}^{N}}\frac{|u(y)|^{p}}
{|x|^{\alpha}|x-y|^{\mu}|y|^{\alpha}}dy\right)|u|^{p}dx.
$$
\end{lem}

\begin{proof}
Define a cut-off function $\varphi(x)\in C_{c}^{\infty}(\mathbb{R}^{N})$ with $0\leq\varphi(x)\leq1$ and $\varphi(x)=1$ in $B_{1}(0)$.
For $0<\lambda<\infty$, we multiply equation \eqref{SWE} by $\psi_{u,\lambda}(x)=\varphi(\lambda x)x\cdot\nabla u(x)$ and integrate over $\mathbb{R}^{N}$,
$$
\int_{\mathbb{R}^{N}}\nabla u\nabla\psi_{u,\lambda}+u\psi_{u,\lambda}dx
=\int_{\mathbb{R}^{N}}\left(\int_{\mathbb{R}^{N}}\frac{|u(y)|^{p}}
{|x|^{\alpha}|x-y|^{\mu}|y|^{\alpha}}dy\right)|u|^{p-2}u\psi_{u,\lambda}dx.
$$
It had been proved in \cite[Proposition 3.1]{MS1} that
\begin{equation}\label{ABC}
\lim_{\lambda\rightarrow0}\int_{\mathbb{R}^{N}}\nabla u\nabla\psi_{u,\lambda}dx
=-\frac{N-2}{2}\int_{\mathbb{R}^{N}}|\nabla u|^{2}dx
\end{equation}
and
\begin{equation}\label{TGB}
\lim_{\lambda\rightarrow0}\int_{\mathbb{R}^{N}}u\psi_{u,\lambda}dx
=-\frac{N}{2}\int_{\mathbb{R}^{N}}|u|^{2}dx.
\end{equation}
Define a function
$$
v(x)=\frac{|u(x)|}{|x|^{\frac{\alpha}{p}}},
$$
then
$$
\frac{x\cdot\nabla u(x)}{|x|^{\frac{\alpha}{p}}}=x\cdot\nabla v(x)+\frac{\alpha}{p}v(x).
$$
Hence we have
\begin{equation}\nonumber
\begin{aligned}
&\left(\int_{\mathbb{R}^{N}}\frac{|u(y)|^{p}}
{|x|^{\alpha}|x-y|^{\mu}|y|^{\alpha}}dy\right)|u|^{p-2}u\psi_{u,\lambda}\\
&\hspace{6mm}=\left(\int_{\mathbb{R}^{N}}\frac{|v(y)|^{p}}
{|x-y|^{\mu}}dy\right)|v|^{p-2}v\psi_{v,\lambda}+\frac{\alpha}{p}\left(\int_{\mathbb{R}^{N}}\frac{|v(y)|^{p}}
{|x-y|^{\mu}}dy\right)|v|^{p}\varphi(\lambda x).\\
\end{aligned}
\end{equation}
Again with the result in \cite[Proposition 3.1]{MS1} that
$$
\lim_{\lambda\rightarrow0}\int_{\mathbb{R}^{N}}\left(\int_{\mathbb{R}^{N}}\frac{|v(y)|^{p}}
{|x-y|^{\mu}}dy\right)|v|^{p-2}v\psi_{v,\lambda}dx
=-\frac{2N-\mu}{2p}\int_{\mathbb{R}^{N}}\left(\int_{\mathbb{R}^{N}}\frac{|v(y)|^{p}}
{|x-y|^{\mu}}dy\right)|v|^{p}dx
$$
and the fact that, by Dominated Convergence Theorem,
$$
\lim_{\lambda\rightarrow0}\int_{\mathbb{R}^{N}}\left(\int_{\mathbb{R}^{N}}\frac{|v(y)|^{p}}
{|x-y|^{\mu}}dy\right)|v|^{p}\varphi(\lambda x)dx
=\left(\int_{\mathbb{R}^{N}}\frac{|v(y)|^{p}}
{|x-y|^{\mu}}dy\right)|v|^{p}dx,
$$
we conclude that
\begin{equation}\label{DEF}
\begin{aligned}
&\lim_{\lambda\rightarrow0}\int_{\mathbb{R}^{N}}\left(\int_{\mathbb{R}^{N}}\frac{|u(y)|^{p}}
{|x|^{\alpha}|x-y|^{\mu}|y|^{\alpha}}dy\right)|u|^{p-2}u\psi_{u,\lambda}dx\\
&=-\frac{2N-2\alpha-\mu}{2p}\int_{\mathbb{R}^{N}}\left(\int_{\mathbb{R}^{N}}\frac{|u(y)|^{p}}
{|x|^{\alpha}|x-y|^{\mu}|y|^{\alpha}}dy\right)|u|^{p}dx.
\end{aligned}
\end{equation}
Taking \eqref{ABC}, \eqref{TGB} and \eqref{DEF} into account, we know
$$
\frac{N-2}{2}\int_{\mathbb{R}^{N}}|\nabla u|^{2}dx+\frac{N}{2}\int_{\mathbb{R}^{N}}|u|^{2}dx
=\frac{2N-2\alpha-\mu}{2p}\int_{\mathbb{R}^{N}}\left(\int_{\mathbb{R}^{N}}\frac{|u(y)|^{p}}
{|x|^{\alpha}|x-y|^{\mu}|y|^{\alpha}}dy\right)|u|^{p}dx.
$$
\end{proof}
{\flushleft\bf{Proof of Theorem \ref{Po}}.}
Once we have the theorem \ref{QWE}, we can multiply equation \eqref{SWE} by $u$ and integrate by parts,
$$
\int_{\mathbb{R}^{N}}|\nabla u|^{2}dx+\int_{\mathbb{R}^{N}}|u|^{2}dx
=\int_{\mathbb{R}^{N}}\left(\int_{\mathbb{R}^{N}}\frac{|u(y)|^{p}}
{|x|^{\alpha}|x-y|^{\mu}|y|^{\alpha}}dy\right)|u|^{p}dx.
$$
Theorem \eqref{QWE} implies
$$
\left(\frac{N-2}{2}-\frac{2N-2\alpha-\mu}{2p}\right)\int_{\mathbb{R}^{N}}|\nabla u|^{2}dx+\left(\frac{N}{2}-\frac{2N-2\alpha-\mu}{2p}\right)\int_{\mathbb{R}^{N}}|u|^{2}dx
=0.
$$
If $p\geq\frac{2N-2\alpha-\mu}{N-2}$ or $p\leq\frac{2N-2\alpha-\mu}{N}$, then $u\equiv0$.$\hfill{} \Box$

\subsection{Regularity} The appearance of nonlocality and singular weight brings the main difficulty in proving the regularity of the solutions. Inspired by the techniques in \cite{LZ, MS1}, we are able to prove the following lemma.

\begin{lem}\label{INT22}
 Assume that $N\geq3$, $\alpha\geq0$, $0<\mu<N$, $0<2\alpha+\mu\leq N$ and $2-\frac{2\alpha+\mu}{N}<p<\frac{2N-2\alpha-\mu}{N-2}$. Let $u\in H^{1}(\mathbb{R}^N)$ be a positive solution of \eqref{SWE}.
Then $u\in W^{2,s}(\mathbb{R}^N)$ for any $s>1$ and $u\in C^\infty(\R^N-\{0\})$.
\end{lem}
\begin{proof}
Since $u\in H^{1}(\mathbb{R}^N)$, $u\in L^{s_{0}}(\mathbb{R}^N)$ with
$$
s_0 =\frac{2Np}{2N-\mu-2\alpha}.
$$
Following the methods in \cite{MS2}, we set $\overline{s}_0 = \underline{s}_0
= s_0$ and assume that $u\in L^{s}(\mathbb{R}^N)$ for every $s\in[\underline{s}_n, \overline{s}_n]$.
Let $Tf(x)=\displaystyle\int_{\mathbb{R}^{N}}\frac{f(y)}
{|x|^{\alpha}|x-y|^{\mu}|y|^{\beta}}dy$, then the weighted Hardy-Littlewood-Sobolev inequality can be written in the form of
$$
|Tf|_{l}=\sup_{|h|_{s}=1}\langle Tf,h\rangle\leq C|f|_{r},
$$
where $\frac{1}{r}+\frac{\alpha+\beta+\mu}{N}=1+\frac{1}{l}$, $\frac{1}{l}+\frac{1}{s}=1$.
Thus, if
$$
\frac{1}{t}=\frac{p}{s}-\frac{N-\mu-2\alpha}{N}>0,
$$
then $\displaystyle\int_{\mathbb{R}^{N}}\frac{|u(y)|^{p}}
{|x|^{\alpha}|x-y|^{\mu}|y|^{\alpha}}dy\in L^{t}(\mathbb{R}^N)$. Further, if
$$
\frac{1}{r}=\frac{2p-1}{s}-\frac{N-\mu-2\alpha}{N}<1,
$$
then $\left(\displaystyle\int_{\mathbb{R}^{N}}\frac{|u(y)|^{p}}
{|x|^{\alpha}|x-y|^{\mu}|y|^{\alpha}}dy\right)|u|^{p-2}u\in L^{r}(\mathbb{R}^N)$. Consequently, if
$$
\frac{N-\mu-2\alpha}{Np}<\frac{1}{s}<\frac{2N-\mu-2\alpha}{N(2p-1)},
$$
by the classical Calder\'{o}n-Zygmund $L^p$ regularity estimates \cite[Chapter 9]{GT}, we know $u\in W^{2,r}(\mathbb{R}^N)$. Since $s\in[\underline{s}_n, \overline{s}_n]$, we know that $u\in W^{2,r}(\mathbb{R}^N)$, for every $r > 1$ such that
$$
\frac{2p-1}{\overline{s}_n}-\frac{N-\mu-2\alpha}{N}<\frac{1}{r}<\frac{2p-1}{\underline{s}_n}-\frac{N-\mu-2\alpha}{N},\
\hbox{and}\
\frac{1}{r}>\frac{N-\mu-2\alpha}{N}\left(1-\frac{1}{p}\right).
$$
By the Sobolev embedding theorem, we know $u\in L^{s}(\mathbb{R}^N)$ provided that
$$
\frac{2p-1}{\overline{s}_n}-\frac{N-\mu-2\alpha+2}{N}<\frac{1}{s}<\frac{2p-1}{\underline{s}_n}-\frac{N-\mu-2\alpha}{N},\
\hbox{and}\
\frac{1}{s}>\frac{N-\mu-2\alpha}{N}\left(1-\frac{1}{p}\right)-\frac{2}{N}.
$$
Following the iterative steps in \cite[Proposition 4.1]{MS2}, we know
\begin{equation}\label{RegP1}
u\in L^s(\R^N), \ \ \hbox{if}\ \ s\geq1 \ \ \hbox{and}\ \ \frac{1}{s}>\frac{N-\mu-2\alpha}{N}\left(1-\frac{1}{p}\right)-\frac{2}{N},
\end{equation}
and
\begin{equation}\label{RegP2}
u\in W^{2,r}(\mathbb{R}^N), \ \ \hbox{if}\ \ r>1 \ \ \hbox{and}\ \ \frac{1}{r}>\frac{N-\mu-2\alpha}{N}\left(1-\frac{1}{p}\right).
\end{equation}

We can prove more regularity of $u$, that is, for every $r >1$, $u\in W^{2,r}(\mathbb{R}^N)$.

In fact, since
$\frac{N-\mu-2\alpha}{N}\left(1-\frac{1}{p}\right)-\frac{2}{N}<\frac{N-\mu-2\alpha}{Np}$, by \eqref{RegP1} we know that $\displaystyle\int_{\mathbb{R}^{N}}\frac{|u(y)|^{p}}
{|x|^{\alpha}|x-y|^{\mu}|y|^{\alpha}}dy\in L^{\infty}(\mathbb{R}^N)$. If $r_0\in(1,\infty)$ is defined by $\frac{1}{r_0}=\frac{N-\mu-2\alpha}{N}\left(1-\frac{1}{p}\right)$, by \eqref{RegP2}, $u\in W^{2,r}(\mathbb{R}^N)$ for every $r\in(1,r_0)$.
Assume now that $u\in W^{2,r}(\mathbb{R}^N)$ for every $r\in(1, r_n)$. By the classical Sobolev embedding
theorem, $u\in L^{s}(\mathbb{R}^N)$ for every $s\in[1,\infty)$ such that
$$
\frac{1}{s}>\frac{1}{r_n}-\frac{2}{N}.
$$
Hence, $\left(\displaystyle\int_{\mathbb{R}^{N}}\frac{|u(y)|^{p}}
{|x|^{\alpha}|x-y|^{\mu}|y|^{\alpha}}dy\right)|u|^{p-2}u\in L^{r}(\mathbb{R}^N)$ for every $r\in(1,\infty)$ such that
$$
p-1 >\frac{1}{r}>(p-1)\left(\frac{1}{r_n}-\frac{2}{N}\right).
$$
By the classical
Calder\'{o}n-Zygmund theory \cite[Chapter 9]{GT}, $u\in W^{2,r}(\mathbb{R}^N)$ for every $r\in(1,\infty)$ such that
$$
\frac{1}{r}>(p-1)\left(\frac{1}{r_n}-\frac{2}{N}\right).
$$
If $r_n\geq \frac N2$, we are done. Otherwise set
$$
\frac{1}{r_{n+1}} =(p-1)\left(\frac{1}{r_n}-\frac{2}{N}\right).
$$
Following the arguments in \cite[Proposition 4.1]{MS2} again, we obtain that
$$
\frac{1}{r_{n+1}}<\frac{1}{r_n},
$$
and then the conclusion is again reached after a finite number of steps.

Apply the Morrey-Sobolev embedding, we know that $u\in C_{loc}^{k,\lambda}(\mathbb{R}^N)$ for
$k\in\{0, 1\}$. Now, for any $x\in\mathbb{R}^{N}-\{0\}$, there exists $r<\frac{|x|}{2}$ such that
$$
\int_{\mathbb{R}^{N}}\frac{|u(y)|^{p}}{|x-y|^{\mu}|y|^{\alpha}}dy=\int_{B_{2r}(x)}\frac{|u(y)|^{p}}{|x-y|^{\mu}|y|^{\alpha}}dy+\int_{\mathbb{R}^{N}-B_{2r}(x)}\frac{|u(y)|^{p}}{|x-y|^{\mu}|y|^{\alpha}}dy.
$$
As the estimate in \eqref{CDE}, for any $\delta<N-\mu$,
\begin{equation}\label{HY}
\int_{B_{2r}(x)}\frac{|u(y)|^{p}}{|x-y|^{\mu}|y|^{\alpha}}dy\in C^{\delta}(\mathbb{R}^{N}-\{0\}).
\end{equation}
Next we show that
$$
\int_{\mathbb{R}^{N}-B_{2r}(x)}\frac{|u(y)|^{p}}{|x-y|^{\mu}|y|^{\alpha}}dy\in C^{\infty}(\mathbb{R}^{N}-\{0\}).
$$
We denote
$$
\psi(x)=\int_{\mathbb{R}^{N}}\frac{|u(y)|^{p}}{|x-y|^{\mu}|y|^{\alpha}}\chi_{\{\mathbb{R}^{N}-B_{2r}(x)\}}dy,
$$
for any small $t<r$, $0<\theta<1$ and $e_{i}$ is the unit $i$th vector,
\begin{equation}\nonumber
\begin{aligned}
\left|\frac{\psi(x+te_{i})-\psi(x)}{t}\right|
&\leq\frac{1}{|t||y|^{\alpha}}\int_{\mathbb{R}^{N}}\left|\frac{|u(y)|^{p}}{|x+te_{i}-y|^{\mu}}\chi_{\{\mathbb{R}^{N}-B_{2r}(x+te_{i})\}}-\frac{|u(y)|^{p}}{|x-y|^{\mu}}\chi_{\{\mathbb{R}^{N}-B_{2r}(x)\}}\right|dy\\
&\leq C\int_{B_{r}(0)}\frac{|u(y)|^{p}}{|x-y|^{\mu+1}|y|^{\alpha}}dy
+C\int_{\mathbb{R}^{N}-B_{r}(x)-B_{r}(0)}\frac{|u(y)|^{p}}{|x-y|^{\mu+1}|y|^{\alpha}}dy.\\
\end{aligned}
\end{equation}
The boundedness of $u(x)$ implies
\begin{equation}\label{KJ}
\int_{B_{r}(0)}\frac{|u(y)|^{p}}{|x-y|^{\mu+1}|y|^{\alpha}}dy
\leq \frac{1}{r^{\mu+1}}||u||^{p}_{L^{\infty}(\mathbb{R}^{N})}\int_{B_{r}(0)}\frac{1}{|y|^{\alpha}}dy<\infty.
\end{equation}
We also have
\begin{equation}\label{LK}
\int_{\mathbb{R}^{N}-B_{r}(x)- B_{r}(0)}\frac{|u(y)|^{p}}{|x-y|^{\mu+1}|y|^{\alpha}}dy
\leq\frac{1}{r^{\mu+1+\alpha}}||u||^{p}_{L^{p}(\mathbb{R}^{N})}<\infty.
\end{equation}
Hence \eqref{KJ}, \eqref{LK} and the Lebesgue's Dominated Convergence Theorem implies $\displaystyle\psi(x)\in C^{1}(\mathbb{R}^{N}-\{0\})$.
Continuing this process, we have $\displaystyle\psi(x)\in C^{\infty}(\mathbb{R}^{N}-\{0\})$.
On the other hand,  $\frac{|u|^p}{|x|^{\alpha}}\in C_{loc}^{k,\lambda}(\mathbb{R}^N)$.
Therefore $\left(\displaystyle\int_{\mathbb{R}^{N}}\frac{|u(y)|^{p}}
{|x|^{\alpha}|x-y|^{\mu}|y|^{\alpha}}dy\right)|u|^{p-2}u\in  C_{loc}^{k,\lambda}(\mathbb{R}^N-\{0\})$ and by the classical Schauder regularity estimates \cite[Theorem 4.6]{GT}, $u\in  C_{loc}^{2,\lambda}(\mathbb{R}^N-\{0\})$. Continuing this process, we obtain the conclusion.
\end{proof}

\subsection{Symmetry}

The Bessel potential $B_{\tau}$ is denoted as
$$
B_{\tau}=(I-\Delta)^{-\frac{\tau}{2}},~~\tau>0.
$$
$\mathcal{F}$ is defined by the Fourier transform
$$
\mathcal{F}(B_{\tau}(f))=(1+4\pi^{2}|x|^{2})^{-\frac{\tau}{2}}\mathcal{F}(f),~~f\in H^{1}(\mathbb{R}^N).
$$
Particularly, when $\alpha=2$, the Bessel potential is the inverse operator of $I-\Delta$ in the space $H^{1}(\mathbb{R}^N)$.

If we consider the nonlinear equation:
\begin{equation}\label{A}
(I-\Delta)^{\frac{\tau}{2}}u=f(x),~~\tau>0.
\end{equation}
It's well-known that equation \eqref{A} is equivalent to
$$
u=g_{\tau}\ast f(x),~~\tau>0.
$$
Where $g_{\tau}$ is the Bessel Kernel defined by
$$
g_{\tau}(x)=\frac{1}{(4\pi)^{\frac{\tau}{2}}\Gamma(\frac{\tau}{2})}
\int_{0}^{\infty}e^{\frac{-\pi|x|^{2}}{t}-\frac{t}{4\pi}}t^{-\frac{N-\tau}{2}-1}dt.
$$
We introduce an important estimate of the Bessel Kernel that will play a crucial role in the rest of the arguments, see \cite{Z}.
\begin{lem}\label{lem1}
For any Bessel Kernel $g_{\tau}(x)$, there holds for any $0<\delta\leq\tau$,
$$
g_{\tau}(x)\leq\frac{C_{1}}{|x|^{N-\tau}e^{C_{2}|x|}}\leq\frac{C}{|x|^{N-\delta}}.
$$
Where $C$, $C_{1}$, $C_{2}$ are fixed constant.
\end{lem}
Hence we can rewrite \eqref{SWE} into an equivalent integral system
\begin{equation}\label{B}
\begin{cases}
u(x)=\displaystyle\int_{\mathbb{R}^{N}}g_{2}(x-y)v(y)u(y)^{p-1}dy,\vspace{3mm}\\
v(x)=\displaystyle\int_{\mathbb{R}^{N}}\frac{u(y)^{p}}{|x|^{\alpha}|x-y|^{\mu}|y|^{\alpha}}dy.
\end{cases}
\end{equation}
Moreover, we present the following symbols:
$$
x^{\lambda}=(2\lambda-x_{1},...,x_{N});~~u(x^{\lambda})=u_{\lambda}(x);~~v(x^{\lambda})=v_{\lambda}(x).
$$
Based on the symbols, we specify some planes:
$$
\Sigma_{\lambda}=\{x\in\mathbb{R}^{N}:x_{1}<\lambda\};~~T_{\lambda}=\{x\in\mathbb{R}^{N}:x_{1}=\lambda\};
$$
$$
\Sigma^{u}_{\lambda}=\{x\in\Sigma_{\lambda}|u(x)>u(x^{\lambda})\};~~\Sigma^{v}_{\lambda}=\{x\in\Sigma_{\lambda}|v(x)>v(x^{\lambda})\}.
$$
\begin{lem}\label{lem3}
For any solution $(u, v)$ of system \eqref{B}, we have
\begin{equation}\label{C}
u(x)-u_{\lambda}(x)=\int_{\Sigma_{\lambda}}\left[g_{2}(x-y)-g_{2}(x^{\lambda}-y)\right]\left[v(y)u(y)^{p-1}-v(y^{\lambda})u(y^{\lambda})^{p-1}\right]dy.
\end{equation}
\end{lem}
\begin{proof}
Direct computing shows
\begin{equation}\nonumber
\begin{aligned}
u(x)
&=\int_{\Sigma_{\lambda}}g_{2}(x-y)v(y)u(y)^{p-1}dy
+\int_{\Sigma_{\lambda}}g_{2}(x-y^{\lambda})v(y^{\lambda})u(y^{\lambda})^{p-1}dy.
\end{aligned}
\end{equation}
Hence we have
$$
u(x)-u_{\lambda}(x)=\int_{\Sigma_{\lambda}}\left[g_{2}(x-y)-g_{2}(x^{\lambda}-y)\right]\left[v(y)u(y)^{p-1}-v(y^{\lambda})u(y^{\lambda})^{p-1}\right]dy.
$$
\end{proof}
Then we have a similar result for $v$ which the proof is similar to Lemma \ref{lem31}, hence we omit it.
\begin{lem}\label{lem4}
For any solution $(u, v)$ of system \eqref{B}, we have
\begin{equation}\label{D}
\begin{aligned}
&v(x)-v_{\lambda}(x)\\
&=\int_{\Sigma_{\lambda}}\frac{1}{|x-y|^{\mu}}\left(\frac{u(y)^{p}}{|x|^{\alpha}|y|^{\alpha}}-\frac{u(y^{\lambda})^{p}}{|x^{\lambda}|^{\alpha}|y^{\lambda}|^{\alpha}}\right)
+\frac{1}{|x^{\lambda}-y|^{\mu}}\left(\frac{u(y^{\lambda})^{p}}{|x|^{\alpha}|y^{\lambda}|^{\alpha}}-\frac{u(y)^{p}}{|x^{\lambda}|^{\alpha}|y|^{\alpha}}\right)dy.
\end{aligned}
\end{equation}
\end{lem}
We should first realize that for all $s>1$, $u\in L^{s}$ from Lemma \ref{INT22}. Based on the preliminary work, for the case $p\neq2$, we can get the following estimate:
\begin{lem}\label{lem5}
Suppose that $\alpha\geq0$, $0<\mu<N$, $2\alpha+\mu\leq3$ if $N=3$ while $2\alpha+\mu<4$ if $N\geq4$ and $2<p<\frac{2N-2\alpha-\mu}{N-2}$. For any $\lambda<0$, $
0<\delta<\min\left\{\frac{2\alpha+\mu}{2},2N(p-2),2\right\}$ and $q>\frac{N}{N+\delta-2\alpha-\mu}$, there exists a constant $C>0$ such that:
\begin{equation}\nonumber
\begin{aligned}
&||u-u_{\lambda}||_{L^{q}({\Sigma^{u}_{\lambda}})}\\
&\leq C\left[||u||^{p}_{L^{\frac{2Np}{\delta+2(N-2\alpha-\mu)}}(\mathbb{R}^{N})}
||u||^{p-2}_{L^{\frac{2N(p-2)}{\delta}}(\Sigma^{u}_{\lambda})}
+||u||^{p-1}_{L^{q(p-1)}(\Sigma^{v}_{\lambda})}||u||^{p-1}_{L^{\frac{Nq(p-1)}{(N+\delta-2\alpha-\mu)q-N}}(\Sigma^{u}_{\lambda})}\right]||u-u_{\lambda}||_{L^{q}(\Sigma^{u}_{\lambda})}.
\end{aligned}
\end{equation}
\end{lem}
\begin{proof}
We divide the domain of integration into four disjoint parts as
$$
\Sigma_{\lambda}=\{\Sigma^{u}_{\lambda}\cap\Sigma^{v}_{\lambda}\}\cup\{\Sigma^{u}_{\lambda}-\Sigma^{v}_{\lambda}\}\cup\{\Sigma^{v}_{\lambda}-\Sigma^{u}_{\lambda}\}\cup
\{\Sigma_{\lambda}-\Sigma^{u}_{\lambda}-\Sigma^{v}_{\lambda}\}.
$$
In $\Sigma^{u}_{\lambda}\cap\Sigma^{v}_{\lambda}$, we can apply the Mean Value Theorem to get
\begin{equation}\nonumber
\begin{aligned}
vu^{p-1}-v_{\lambda}u_{\lambda}^{p-1}
&\leq(p-1)vu^{p-2}(u-u_{\lambda})+u^{p-1}(v-v_{\lambda}).
\end{aligned}
\end{equation}
In $\Sigma^{u}_{\lambda}-\Sigma^{v}_{\lambda}$, by the Mean Value Theorem again, we know
\begin{equation}\nonumber
\begin{aligned}
vu^{p-1}-v_{\lambda}u_{\lambda}^{p-1}
&=(p-1)vu^{p-2}(u-u_{\lambda}).
\end{aligned}
\end{equation}
In $\Sigma^{v}_{\lambda}-\Sigma^{u}_{\lambda}$,
\begin{equation}\nonumber
\begin{aligned}
vu^{p-1}-v_{\lambda}u_{\lambda}^{p-1}
&=u^{p-1}(v-v_{\lambda}).
\end{aligned}
\end{equation}
In $\Sigma_{\lambda}-\Sigma^{u}_{\lambda}-\Sigma^{v}_{\lambda}$,
\begin{equation}\nonumber
\begin{aligned}
vu^{p-1}-v_{\lambda}u_{\lambda}^{p-1}&\leq0.
\end{aligned}
\end{equation}
Hence Lemma \ref{lem1} and Lemma \ref{lem3} imply that for any $0<\delta\leq2$, there exists a constant $C>0$ such that when $x\in\Sigma_{\lambda}$,
\begin{equation}\nonumber
\begin{aligned}
u(x)-u_{\lambda}(x)
&\leq\int_{\Sigma_{\lambda}}\left[g_{2}(x-y)-g_{2}(x^{\lambda}-y)\right]
\left[(p-1)vu^{p-2}(u-u_{\lambda})\chi_{\{\Sigma^{u}_{\lambda}\}}+u^{p-1}(v-v_{\lambda})\chi_{\{\Sigma^{v}_{\lambda}\}}\right]dy\\
&\leq C\int_{\Sigma^{u}_{\lambda}}\frac{vu^{p-2}(u-u_{\lambda})}{|x-y|^{N-\delta}}
dy+C\int_{\Sigma^{v}_{\lambda}}\frac{u^{p-1}(v-v_{\lambda})}{|x-y|^{N-\delta}}dy.
\end{aligned}
\end{equation}
By Lemma \ref{lem4} we know
\begin{equation}\nonumber
\begin{aligned}
v(x)-v_{\lambda}(x)
&\leq\int_{\Sigma^{u}_{\lambda}}\frac{u(y)^{p}-u(y^{\lambda})^{p}}{|x|^{\alpha}|x-y|^{\mu}|y|^{\alpha}}dy,  \ \ x\in\Sigma_{\lambda}.
\end{aligned}
\end{equation}
With the Mean Value Theorem, we see that
\begin{equation}\label{H}
v(x)-v_{\lambda}(x)\leq p\int_{\Sigma^{u}_{\lambda}}\frac{u(y)^{p-1}[u(y)-u(y^{\lambda})]}{|x|^{\alpha}|x-y|^{\mu}|y|^{\alpha}}dy.
\end{equation}
Choose $\delta$ such that
$$
0<\delta<\min\left\{\frac{2\alpha+\mu}{2},2N(p-2),2\right\},
$$
then for any $q>\frac{N}{N+\delta-2\alpha-\mu}$, by the H\"{o}lder inequality and the weighted HLS inequality we find that
\begin{equation}\label{G}
\begin{aligned}
||v-v_{\lambda}||_{L^{\frac{N}{\delta}}(\Sigma^{v}_{\lambda})}
&\leq C||u^{p-1}(u-u_{\lambda})||_{L^{\frac{N}{N+\delta-2\alpha-\mu}}(\Sigma^{u}_{\lambda})}\\
&\leq C||u||^{p-1}_{L^{\frac{Nq(p-1)}{(N+\delta-2\alpha-\mu)q-N}}(\Sigma^{u}_{\lambda})}||u-u_{\lambda}||_{L^{q}(\Sigma^{u}_{\lambda})}.
\end{aligned}
\end{equation}
In virtue of the H\"{o}lder inequality and the HLS inequality again, it follows that
\begin{equation}\nonumber
\begin{aligned}
\left\|\int_{\Sigma^{u}_{\lambda}}\frac{vu^{p-2}(u-u_{\lambda})}{|x-y|^{N-\delta}}
dy\right\|_{L^{q}(\Sigma^{u}_{\lambda})}
&\leq C||vu^{p-2}(u-u_{\lambda})||_{L^{\frac{Nq}{N+\delta q}}(\Sigma^{u}_{\lambda})}\\
&\leq C||u||^{p}_{L^{\frac{2Np}{\delta+2(N-2\alpha-\mu)}}(\mathbb{R}^{N})}
||u||^{p-2}_{L^{\frac{2N(p-2)}{\delta}}(\Sigma^{u}_{\lambda})}
||u-u_{\lambda}||_{L^{q}(\Sigma^{u}_{\lambda})}.
\end{aligned}
\end{equation}
On the other hand, the H\"{o}lder inequality and \eqref{G} imply that
\begin{equation}\nonumber
\begin{aligned}
\left\|\int_{\Sigma^{v}_{\lambda}}\frac{u^{p-1}(v-v_{\lambda})}{|x-y|^{N-\delta}}dy\right\|_{L^{q}(\Sigma^{u}_{\lambda})}
&\leq C||u^{p-1}(v-v_{\lambda})||_{L^{\frac{Nq}{N+\delta q}}(\Sigma^{v}_{\lambda})}\\
&\leq C||u||^{p-1}_{L^{q(p-1)}(\Sigma^{v}_{\lambda})}||u||^{p-1}_{L^{\frac{Nq(p-1)}{(N+\delta-2\alpha-\mu)q-N}}(\Sigma^{u}_{\lambda})}||u-u_{\lambda}||_{L^{q}(\Sigma^{u}_{\lambda})}.
\end{aligned}
\end{equation}
Hence there exists a constant $C>0$,
\begin{equation}\nonumber
\begin{aligned}
&||u-u_{\lambda}||_{L^{q}({\Sigma^{u}_{\lambda}})}\\
&\leq C\left[||u||^{p}_{L^{\frac{2Np}{\delta+2(N-2\alpha-\mu)}}(\mathbb{R}^{N})}
||u||^{p-2}_{L^{\frac{2N(p-2)}{\delta}}(\Sigma^{u}_{\lambda})}
+||u||^{p-1}_{L^{q(p-1)}(\Sigma^{v}_{\lambda})}||u||^{p-1}_{L^{\frac{Nq(p-1)}{(N+\delta-2\alpha-\mu)q-N}}(\Sigma^{u}_{\lambda})}\right]||u-u_{\lambda}||_{L^{q}(\Sigma^{u}_{\lambda})}.
\end{aligned}
\end{equation}
\end{proof}
For the case $p=2$, Lemma \ref{lem5} should be replace by
\begin{lem}\label{lemABCDEFG}
Suppose that $\alpha\geq0$, $0<\mu<N$, $2\alpha+\mu\leq3$ if $N=3$ while $2\alpha+\mu<4$ if $N\geq4$ and $p=2$. For any $\lambda<0$, $
0<\delta<\min\left\{\frac{2\alpha+\mu}{2},2\right\}$ and $q>\frac{N}{N+\delta-2\alpha-\mu}$, there exists a constant $C>0$ such that:
\begin{equation}\nonumber
\begin{aligned}
&||u-u_{\lambda}||_{L^{q}({\Sigma^{u}_{\lambda}})}\\
&\hspace{6mm}\leq C\left[||v||^{2}_{L^{\frac{2N}{\delta}}(\Sigma^{u}_{\lambda})}
+||u||_{L^{q}(\Sigma^{v}_{\lambda})}||u||_{L^{\frac{Nq}{(N+\delta-2\alpha-\mu)q-N}}(\Sigma^{u}_{\lambda})}\right]||u-u_{\lambda}||_{L^{q}(\Sigma^{u}_{\lambda})}.
\end{aligned}
\end{equation}
\end{lem}
We will first show that, for $\lambda$ is sufficiently negative,
\begin{equation}\label{E}
\begin{aligned}
u(x)\leq u(x^{\lambda}),~~~v(x)\leq v(x^{\lambda}),~~\forall x\in\Sigma_{\lambda}.
\end{aligned}
\end{equation}
Then we can start moving the plane from near $-\infty$ to the right as long as \eqref{E} holds.
\begin{lem}\label{lem6}
Under the assumption of Theorem \ref{thm}, there exists $\lambda_{0}<0$ such that for any $\lambda\leq\lambda_{0}$, $u(x)\leq u(x^{\lambda})$ and $v(x)\leq v(x^{\lambda})$ hold in $\Sigma_{\lambda}$.
\end{lem}
\begin{proof}
The proof is similar to the proof of Lemma \ref{lem61}, hence we omit it.
\end{proof}
We now move the plane $T_{\lambda}$ to the left as long as \eqref{E} holds. If we define
$$
\lambda_{1}=\sup\{~\lambda~|~u(x)\leq u(x^{\mu})~,~v(x)\leq v(x^{\mu})~,~\forall x\in\Sigma_{\mu}~,~\mu\leq\lambda\},
$$
then we must have $\lambda_{1}<\infty$ because of applying a similar argument for $\lambda$ near $+\infty$.
\begin{lem}\label{lem7}
Under the assumption of Theorem \ref{thm}, then for any $\lambda_{1}<0$, there holds $u(x)\equiv u(x^{\lambda_{1}})$ and $v(x)\equiv v(x^{\lambda_{1}})$ in $\Sigma_{\lambda_{1}}$.
\end{lem}
\begin{proof}
The proof is similar to the proof of Lemma \ref{lem71}, hence we omit it.
\end{proof}
{\flushleft\bf{Proof of Theorem \ref{thm}}.}
Similarly, the plane can be moved from $+\infty$ to left, therefore we denote the corresponding parameter $\lambda_{1}^{\prime}$ satisfying
$$
\lambda^{\prime}_{1}=\inf\{~\lambda~|~s(x)\leq s(x^{\mu})~,~t(x)\leq t(x^{\mu})~,~x\in\Sigma^{\prime}_{\mu}~,~\mu\geq\lambda\}.
$$
Where $\Sigma^{\prime}_{\lambda}=\{x\in\mathbb{R}^{N}:x_{1}>\lambda\}$. If $\lambda^{\prime}_{1}>0$, $u(x)$ and $v(x)$ are also radially symmetric about $T_{\lambda^{\prime}_{1}}$.
Hence when $\lambda_{1}=\lambda^{\prime}_{1}\neq0$, we can deduce from Lemma \ref{lem7} that $u(x)\equiv u(x^{\lambda_{1}})$ and $v(x)\equiv v(x^{\lambda_{1}})$ in $\Sigma_{\lambda_{1}}$,
that contradicts to \eqref{D}.
Therefore, we have $\lambda_{1}=\lambda^{\prime}_{1}=0$, $u(x)$ and $v(x)$ are radially symmetric about $T_{0}$.
Since the directions are chosen arbitrarily, we deduce $u(x)$ and $v(x)$ are radially symmetric about origin.$\hfill{} \Box$

\vspace{1cm}
\noindent {\bf Acknowledgements.} \
The authors would like to thank  the anonymous referee
for his/her useful comments and suggestions which help to improve the presentation of the paper greatly.

\end{document}